\documentclass[12pt]{article}
\usepackage{}
\usepackage{multirow}
\usepackage{color}
\usepackage{diagbox}
\usepackage{graphicx}
\usepackage{epsfig}
\usepackage{graphicx}
\usepackage[tbtags]{amsmath}
\usepackage{epsfig}
\usepackage{epstopdf}
\usepackage{amssymb}
\usepackage{bm}

\usepackage{amsmath}
\usepackage{amsthm}
\usepackage{cases}
\usepackage{amsfonts}
\usepackage{mathrsfs}
\usepackage{multirow}
\usepackage{tabularx}
\usepackage{color}
\usepackage{algorithm}
 \usepackage{pgfplots}
\usepackage{grffile}
\usepackage{amsmath}
\newtheorem{thm}{Theorem}[section]
\newtheorem{lem}{Lemma}[section]
\newtheorem{rem}{Remark}[section]
\newtheorem{defn}{Definition}[section]
\newtheorem{pro}{Proposition}[section]

\newtheorem{exm}{Example}[section]
\newtheorem{alg}{Algorithm}[section]
\newtheorem{as}{Assumption}[section]

\numberwithin{equation}{section} \allowdisplaybreaks[4]

\begin{document}
\date{}
\pagestyle{plain}
\title {Reach-avoid semi-Markov decision processes with time-varying obstacles \thanks{Research supported by NSFC (Grant No. 11931018).}}
\author{Yanyun Li and Xianping Guo \thanks{Corresponding author.
Email: mcsgxp@mail.sysu.edu.cn (X.P. Guo).}\\
  \small{School of Mathematics,} \\ \small{Sun Yat-Sen University, Guangzhou, 510275, China}}

\date{}
\maketitle \underline{}
\vskip 0.2 in \noindent
{\bf Abstract:}
We consider the maximal reach-avoid probability to a target in finite horizon for semi-Markov decision processes with time-varying obstacles. Since the variance of the obstacle set, the model \eqref{Model} is non-homogeneous. To overcome such difficulty, we construct a related two-dimensional model \eqref{newModel}, and then prove the equivalence between such reach-avoid probability of the original model and that of the related two-dimensional one. For the related two-dimensional model, we analyze some special characteristics of the equivalent reach-avoid probability. On this basis, we provide a special improved value-type algorithm to obtain the equivalent maximal reach-avoid probability and its $\epsilon$-optimal policy. Then, at the last step of the algorithm, by the equivalence between these two models, we obtain the original maximal reach-avoid probability and its $\epsilon$-optimal policy for the original model.



\setlength{\baselineskip}{0.25in}

\section{Introduction} \label{intro}

Safety and reachability are two of the most fundamental aspects in controlled dynamical systems, which can be modeled by using the framework of
Markov decision processes (MDPs), see \cite{DM22,DEJ11,LL-23,LGG-2023,MZ23}. One of the main objectives in reachability problems for
MDPs, is to maximize the probability of reaching a target set within a given time-horizon from regular states, usually called a reach-avoid probability.  The reach-avoid problem in discrete-time and continuous-time MDPs had been analyzed in \cite{DM22,DEJ11,LL-23,LGG-2023}. Note that the sojourn time at each state in the model analyzed in \cite{LL-23} is exponential distributed, it is natural to consider the reach-avoid problem in the semi-MDPs where the sojourn time is general distributed.

Regarding the reach-avoid problem, the main research objects are the maximal probabilistic reachable set (i.e., a set of states from which the evolution of the system has a reach-avoid probability), a ``yes" or ``no" problem (i.e., whether it is possible to reach the target set in a given time starting from a certain set) and the maximal reach-avoid probability.
For the first one, a method for computing maximal probabilistic reachable set in nondeterministic systems, was presented in \cite{LLWY-22}. For the second one, various methods have been proposed to deal with the ``yes" or ``no" problem, including the ellipsoidal method \cite{L-20}, the polyhedral method \cite{A-03}, and the level set method \cite{J-04}. For the third one, many researchers have studied the problem of calculating the maximal reach-avoid probability in MDPs, see \cite{DM22,DEJ11,LL-23,LGG-2023}.
Different from above, our research is aim to find out the maximal reach-avoid probability in semi-MDPs.
Actually, the reach-avoid probability can be regarded as the probability of an airplane reaching the target location in a safe flying space.

In MDPs, most researchers considered the risk neutral criteria (see \cite{A02,OL96}), risk probability criterion (see \cite{BM04,HG-2020}) and risk-sensitive criterion (see \cite{BR14,CH11}). For the problem of computing
the maximal reach-avoid probability in MDPs, one can refer to~\cite{DM22,DEJ11,LL-23,LGG-2023}. In detail, the existence of an optimal policy of such problem in discrete-time MDPs had been proved in \cite{DEJ11}; the transformation from the reach-avoid probability into an equivalent long-run average reward in discrete-time MDPs, had been given in \cite{DM22}; A novel state-classification-based PI approach of computing the maximal reach-avoid probability in discrete-time MDPs, had been presented in \cite{LGG-2023}, which solved the non-uniqueness problem of its solution to the original optimality equation; in continuous-time MDPs, \cite{LL-23} found that the maximal reach-avoid probability can be dealt with under the embedded Markov chains that can be regarded as a special discrete-time MDP in the finite state space case (see \cite{LGG-2023}), and in a controlled branching process (i.e., a special MDP), obtained an algorithm of computing minimal extinction probability (i.e., minimal reach-avoid probability with the target set being a single point set $\{0\}$). However,
the problem of computing the maximal reach-avoid probability mentioned above
is defined by a fixed obstacle set.

In this paper, we continue this line of research by studying the maximal reach-avoid probability with time-varying obstacles in semi-MDPs.
The main contributions of this study are as follows:
\begin{enumerate}
\item Different from \cite{DM22,DEJ11,LL-23,LGG-2023}, since there are time-varying obstacles in semi-MDPs, we can not determine which situation of transformation occurs at every step under the stochastic kernel $Q$. To overcome this difficulty, we introduce a transferred method that is similar with the method of enlarging its state space mentioned in \cite{NU-17}, and then show that the reach-avoid probability in the original model \eqref{Model} is equivalent to the corresponding reach-avoid probability in the equivalent semi-Markov model \eqref{newModel}, see Theorem \ref{thm3.1}. The main advantage of such transferred method is that one can deal with the problem caused by the time-varying obstacles, and transfer the non-homogeneous model \eqref{Model} into the homogeneous model \eqref{newModel}.

\item  We present an algorithm of calculating the maximal reach-avoid probability and its $\epsilon$-optimal policy of the original model \eqref{Model}.  More precisely, the equivalent maximal reach-avoid probability and its equivalent $\epsilon$-optimal policy
	is provided in Algorithm \ref{al-finite}, and then, by the transferred result (Theorem \ref{thm3.1}) and Lemma \ref{lem3.1}, the maximal reach-avoid probability and its $\epsilon$-optimal policy in the original model \eqref{Model} can be transferred from \eqref{newModel}, see Step 4 in Algorithm \ref{al-finite}. Especially, as one can see in Steps 1-3 of Algorithm \ref{al-finite}, the transition steps of the original model \eqref{Model} are supplemented to the state of the equivalent model \eqref{newModel} by $\tilde{Q}$, which overcomes the non-homogeneity caused by the varying obstacle set. In Steps 1-3, we only need to calculate one value function beginning with $k$'th decision epoch at some iteration, and with each additional iteration, we obtain the corresponding value function starting the previous decision epoch. Finally, in Step 4, we obtain the final optimal value function starting the first decision epoch, which is the maximal reach-avoid probability in model \eqref{newModel}.
	
	\item In addition, we analyze several special properties of the equivalent model \eqref{newModel} in Theorem \ref{thm useful}, whose function of the varying-time obstacle set is presented in Remark \ref{rem-2}. Moreover, via an example given in Section \ref{sec-5}, we find a special law of the varying-time obstacle set given in Remark \ref{rem-3}.
\end{enumerate}

This paper unfold as follows: In Section \ref{sec-2}, we briefly introduce the reach-avoid problem in semi-MDPs. Section \ref{sec-3} contains the transferred method of transferring the non-homogeneous model to the homogeneous model. The special properties of the equivalent model, the uniqueness of the solution to optimality equation, the existence of an optimal policy and an algorithm of the maximal reach-avoid probability and its $\epsilon$-optimal policy, are provided in Section \ref{sec-4}. Finally, an example about the filght of the plane is presented in Section \ref{sec-5}.

\section{Description of reach-avoid problems in semi-MDPs}\label{sec-2}

The reach-avoid problem under semi-Markov decision processes in a finite horizon  $\mathbb{R}_T:=[0,T]$ with $T < \infty$ considered in this paper, is formulated by
\begin{equation}\label{Model}
	\{E, (B_n:n\geq 0), C, (A(x)\subset A: x \in E), Q(\cdot,\cdot| x, a)\},
\end{equation}
where the five elements are explained as below:

(1)\ $E$ is a Borel state space, that is, a Borel subset of a complete and separable metric space, denoting the set of all observable states of a system, with the Borel $\sigma$-algebra $\mathcal{B}(E)$.
	
(2)\  $B_n\! \in\! \mathcal{B}(E)\ (n\geq 0)$ and $C\! \in\! \mathcal{B}(E)$, satisfy that $B_n\! \cap\! C=\emptyset$ and $E\! \setminus \!(B_n\! \cup\! C)\neq \emptyset$. Note that $B_n$ can be regarded as a cemetery set at the $n$'th step, and $C$ as a fixed target set. For example, a plane flies to a target place and it will meet different obstacles during its flight route, see examples in \cite{IAC-05,KJ-11}.

(3)\  $A(x)$ is a finite set of actions admissible at state $x\! \in\! E$ and $A=\bigcup\limits_{x\in E}A(x)$.

(4)\ $Q(\cdot,\cdot| x, a)\ (x\!\in\! E, a\!\in\! A(x))$, is the one-step transition mechanism of the system. By letting $K:=\{(x,a)| x\! \in\! E,\ a\! \in\! A(x)\}$ be the set of all feasible state-action triple, $Q(\cdot,\cdot| x, a)$ is defined by the semi-Markov kernel $Q$ on $E\times\mathbb{R}_T$, satisfying that: (i)\ for any fixed $D\! \in\! \mathcal{B}(E)$ and $(x,a)\! \in\! K$, $Q(D,\cdot| x, a)$ is a nondecreasing and right-continuous real-valued function with $Q(D, 0| x, a)=\delta_{D}(x)$; (ii)\ for each fixed $t$, $Q(\cdot,t| \cdot, \cdot)$ is a sub-stochastic kernel on $E$ given $K$; and (iii)\ $Q(\cdot, \infty| \cdot, \cdot):=\lim\limits_{t \rightarrow \infty}Q(\cdot,t| \cdot, \cdot)$ is a stochastic kernel on $E$ given $K$. For a fixed pair $(x,a)\! \in\! K$, $Q(\cdot,\cdot| x,a)$ is the joint probability distribution of the sojourn time at state $x$ and the next state.

We now describe the evolution of the finite horizon semi-MDP. Assume that the initial state is $x_0\!\in\! E$ and initial decision epoch is $0$. The decision-maker chooses an action $a_0\!\in\! A(x_0)$. Under action $a_0$, the process remains at state $x_0$ for a random time $s_0$ and then transfers to state $x_1$ according to the transition kernel $Q(\cdot,\cdot |x_0,a_0)$. Then the decision-maker chooses an action $a_1\! \in\! A(x_1)$ and the process transfers into another state $x_2$ after the sojourn time $s_1$ according to the transition kernel $Q(\cdot, \cdot| x_1, a_1)$. At the decision epoch $s_0+\cdots +s_{n-1}$, the decision-maker chooses an action $a_n\! \in\! A(x_n)$. Then, the process stays at state $x_n$ for a random time $s_{n}$ and transfers to state $x_{n+1}$ according to the transition kernel $Q(\cdot,\cdot|x_n,a_n)$.
The process evolves in this way and thus we obtain an admissible history $h_n$ of the semi-MDPs up to
the $n$'th decision epoch, i.e.,
\begin{eqnarray*}
	h_n:=(x_0, a_0, s_0, x_1, a_1, s_1, \cdots, x_{n-1}, a_{n-1}, s_{n-1}, x_n).
\end{eqnarray*}
Denote $H_n$ as  the set of all admissible histories $h_n$ of the process up to the $n$'th decision epoch, where
$H_n$ is endowed with the Borel $\sigma$-algebra.

In many real situations, the time-varying obstacles are objectively existent, and such MDPs with time-varying obstacles can be applied in plane flight system and intelligent traffic system, see \cite{IAC-05,LLWY-22}. Below we give two examples to illustrate the advantage of time-varying obstacles.
\begin{description}
\item[(i)] {\bf Plane flight system:} In plane flight system, the set of time-varying obstacles usually includes ground obstacles (such as buildings, vehicles, etc.), aerial obstacles (such as other aircraft, flocks of birds, etc.), meteorological phenomena (such as turbulence, freezing, wind shear, etc.), and no-fly zones. The existence of these obstacles on the flying route, poses a serious challenge to flight safety and efficiency. So, effective decision-making and planning methods are needed to avoid collisions and ensure safety. Using the MDP model, the controller designs intelligent control policies to avoid collisions. For example, a reward function is defined to penalize collision events while rewarding safe paths.
		
\item[(ii)] {\bf Intelligent traffic system:} In urban traffic, traffic accidents and road construction will lead to temporary closure or restriction of some road sections, forming a changing obstacle area. Based on the MDP with a changing barrier set, the traffic management system can use these affected road sections as a changing barrier set based on real-time traffic condition information, and optimize decision-making such as traffic light duration and vehicle scheduling to improve overall traffic efficiency.
	\end{description}

\begin{exm}\label{exm2.1}
	\rm{ Consider a plane flight traffic system.
A vehicle treated as a mass point, moves with a constant linear speed $v$ on $E$, where the state space is $E:=\{0,1,2,\cdots,m\}$.
Suppose that when the vehicle at the state $i$, the pilot of the vehicle can control its direction by using control stick and pedal, and will choose different actions from  $A(i):=\{\alpha, \beta, \gamma\}$ for all $i \in E \setminus C$ to control the stick and pedal. The vehicle flies to the next state $j \in E \setminus C$ according to the transition kernel with regard to the action selected by the pilot and the current state $i$.
In the flying route of the plane, there are different obstacle sets at different decision epochs. These obstacle sets can be regarded as the birds, cumulonimbus, other planes, drones and high-rise buildings, iron towers, wind turbines, etc. The vehicle is aim to arrive at a destination, that is, a target set $C$. The purpose of the pilot of the vehicle is to avoid the obstacles before reaching the target set $C$.
}
\end{exm}

For convenience of our discussion, we give the concept of policies (decision rules) for the decision-maker to select actions.

\begin{defn}
\rm{
A randomized history-dependent policy is a sequence $\pi=\{\pi_n\!: n \geq 0\}$ of stochastic kernels $\pi_n$ on $A$ given $H_n$ satisfying
\begin{eqnarray*}
	\pi_n(A(x_n)| h_n)=1\ \ \forall \ h_n\! \in\! H_n,\ n  \geq 0.
\end{eqnarray*}
}
\end{defn}	
The set of all randomized history-dependent policies is denoted by $\Pi$.
\begin{defn}
\begin{description}	
\rm{
\item[(i)] A policy $\pi=\{\pi_n\!:n \geq 0\}\in\! \Pi$ is said to be randomized Markov if there is a sequence $\{\psi_n\!:n\geq 0\}$ of stochastic kernels on $A$ given $E$ such that $\psi_n(A(x)|x)=1$ for all
$x\! \in\! E$ and $\pi_n(\cdot|h_n)=\psi_n(\cdot|x_n)$ for every $h_n\! \in\! H_n$ and $n \geq 0$. In this case, the policy $\pi=\{\pi_n:n\geq 0\}$ is rewritten as $\pi=\{\psi_n:n\geq 0\}$.

\item[(ii)]
A randomized Markov policy $\pi=\{\psi_n:n\geq 0\} $ is called randomized stationary Markov if  $\psi_n=\psi$ for all $n \geq 0$. In this case, the policy $\pi=\{\psi,\psi,\cdots\}$ is abbreviated as $\psi$.

\item[(iii)]  A randomized Markov policy $\pi=\{\psi_n: n \geq 0\}$ is called deterministic Markov policy if there exists a sequence of decision functions $\{f_n: n \geq 0\}$ such that $\psi_n(\cdot|x)=\delta_{f_n(x)}(\cdot)$. In this case, the policy $\pi=\{\psi_n: n \geq 0\}$ is denoted as $\pi=\{f_n: n \geq 0\}$.

 A deterministic Markov policy $\pi=\{f_n: n \geq 0\}$ is called stationary deterministic  Markov policy, if there exists a decision function $f$ such that $f_n=f\ (n \geq 0)$. In this case, the policy $\pi=\{f,f,\cdots\}$ is abbreviated by $f$.
	}
	\end{description}
\end{defn}
For convenience, let $\Pi_{rm}$, $\Pi_s$, $\Pi_d$ and $\Pi_{sd}$ denote the set of all randomized Markov policies, the set of all randomized stationary Markov policies, the set of all deterministic Markov policies and the set of all deterministic stationary Markov policies, respectively. Clearly, $\Pi_{sd}  \subset \Pi_s\ (\Pi_d) \subset \Pi_{rm} \subset \Pi$.

Let $(\Omega, \mathcal{F})$ be the measurable space, where
\begin{eqnarray*}
	\Omega=\{(x_0,a_0,s_0,\ldots,x_n,a_n,s_n,\ldots)|\ ( x_n, a_n, s_n)\!\in\!  E\! \times\! A\! \times\! \mathbb{R}_T\ \text{for}\ n\geq 0\},
\end{eqnarray*}
and $\mathcal{F}$ is the corresponding Borel $\sigma$-algebra. Then, we define maps $Z_n$, $A_n$ and $\sigma_n$\ $(n \geq 0)$ on $(\Omega, \mathcal{F})$ as follows: for each $\omega:=(x_0,a_0,s_0,\ldots,x_n,a_n,s_n,\ldots) \in \Omega$,
\begin{eqnarray*}
	\sigma_0(\omega)=0, \ \sigma_n(\omega)=s_0+\cdots+s_{n-1},\ Z_n(\omega)=x_n,\ A_n(\omega)=a_n,
\end{eqnarray*}
where $\sigma_n$ is the $n$'th decision epoch, $Z_n$ and $A_n$ are the state and action chosen at the $n$'th decision epoch, respectively.
Therefore, by the well-known Ioneasu Tulcea theorem \cite{OL96}, for  each $x\! \in\! E$ and $\pi\! \in\! \Pi$, there exists a unique probability measure $P^{\pi}_{x}$ such that, for every $t\! \in\! \mathbb{R}_T$, $D\! \subset\! \mathcal{B}(E)$, $a\! \in\! A$ and $n \geq 0$,
\begin{eqnarray}\label{P-3}
	&&P^{\pi}_{x}(\sigma_0=0,Z_0=x)=1,
	\quad P^{\pi}_{x}(A_{n+1}=a|\ h_n)=\pi_n(a|\ h_n)\\
	&&P^{\pi}_{x}(Z_{n+1}\in D,\sigma_{n+1}-\sigma_{n}\leq t|\ h_n, a_n)=Q(D,t|x_n, a_n).
\end{eqnarray}
Denote $E^{\pi}_{x}$ as the expectation operator associated with $P^{\pi}_{x}$.
To avoid possibility of infinitely decision epochs during a finite horizon $\mathbb{R}_T$, we impose the following basic assumption.

\begin{as}\label{As-1}
\rm{	
	$P^{\pi}_x(\lim\limits_{n \rightarrow\infty}\sigma_n=\infty)=1$ for all $x\!\in\! E$ and $\pi\! \in\! \Pi$.
}
\end{as}
The above assumption is same as Assumption 2.1 in \cite{HGS-11}. Moreover, it is trivially fulfilled in discrete-time MDPs. We suppose that
Assumption \ref{As-1} holds throughout this paper. Although Assumption \ref{As-1} is natural and mild, it is not easy to verify in applications. The following Proposition~\ref{As-2} gives a sufficient condition for Assumption~\ref{As-1} and one can refer Proposition 2.1 in \cite{HG11,HGS-11} for its proof.

\begin{pro}\label{As-2}
	\rm{ Suppose that there exist positive constants $\delta$ and $\epsilon_0$ such that
\begin{eqnarray}\label{As-2eq}
Q(E,\delta|\ x,a) \leq 1-\epsilon_0\quad \text{for\ all} \ x\! \in\! E\! \setminus\! B_0 \ \text{and}\ a\! \in\! A(x).
\end{eqnarray}
Then Assumption~\ref{As-1} holds.}
\end{pro}

Under Assumption \ref{As-1}, we can define an underlying continuous-time state-action process $\{(X_t, \mathfrak{a}_t): t\!\in\! \mathbb{R}_T\}$ by
\begin{eqnarray*}
	X_t=Z_n,\ \ \mathfrak{a}_t=A_n,\ \text{for}\ t\in[\sigma_n,\sigma_{n+1}),\ \ n \geq 0,
\end{eqnarray*}
which is called a finite horizon semi-MDP. It is well-known that semi-MDPs can describe a great variety of real-world situations such as queuing systems and maintenance problems \cite{CO-10,LZ-00,SVA-07}.

To state our reach-avoid problem, let
\begin{eqnarray}\label{def-tau}
	\begin{cases}
	\tau_{_C}:=\inf \{t \geq 0:\ X_t\! \in\! C\}=\inf\{n\geq 0:\ Z_n\! \in\! C\}\ \ (\inf\emptyset=\infty)\\
	\bar{\tau}:=\inf \{\sigma_n \geq 0:\ X_{\sigma_n}\! \in\! B_n\}
\end{cases}
\end{eqnarray}
be  first hitting time on $C$
and first time such that $X_{\sigma_n}\! \in\! B_n$, respectively. In the following, $\bar{\tau}$ is called the cemetery-hitting time.

For a given policy $\pi\! \in\! \Pi$ and an initial state $x$, the probability of reaching $C$ before cemetery-hitting during a finite period time $[0,t]$ for each $t\! \in\! \mathbb{R}_T$, is defined by
\begin{eqnarray}\label{T-criterion}
	G(x,t,\pi):=P^{\pi}_x(\tau_C<\bar{\tau} \wedge t)\quad \text{for\ any}\ (x,t)\! \in\! E\! \times\! \mathbb{R}_T,
\end{eqnarray}
which is usually called the reach-avoid probability (see \cite{DM22,DEJ11}).

\begin{defn}\label{C-def}
		\rm{The set $D \subset E$ is a uniformly-absorbing set if for any $x \in D$ and $a \in A(x)$, $Q(D, \infty| x, a)=1$.
}
\end{defn}
Since $G(x,t,\pi)$ only depends on the evolution of the process before hitting $C$ and the set $C$ is the target set, it is natural to assume that $C$ is a uniformly-absorbing set from now on. It is obvious that $G(x,t, \pi) \equiv 0\ (x \!\in\! B_0)$ and $G(x,t,\pi) \equiv 1\ (x\in C)$, we only need to consider the initial state $x \in E\! \setminus\! (B_0 \cup C)$. Then, define the maximal reach-avoid probability as below: for each  $x\!\in\! E \setminus (B_0 \cup C)$,
\begin{equation}\label{G^*}
	G^*(x,t):=\sup_{\pi \in \Pi}G(x,t,\pi)=\sup_{\pi \in \Pi}P^{\pi}_x(\tau_C<\bar{\tau} \wedge t)\quad  \text{for\ any }\ t\! \in\! \mathbb{R}_T.
\end{equation}

\begin{defn}\label{O-def}
{\rm	\begin{description}
		\item[(i)] A policy $\pi^*\in \Pi$ is called ($T$-horizon) optimal if $G(\cdot,T,\pi^*)=G^*(\cdot,T)$.
	\item[(ii)] A policy $\pi^{\epsilon}\in \Pi$ is called ($T$-horizon) $\epsilon$-optimal if
$|G(\cdot,T,\pi^{\epsilon})-G^*(\cdot,T)|<\epsilon$.
	\end{description}
}
\end{defn}
The main purpose of this paper is to find an optimal policy $\pi^*\in \Pi$ such that
\begin{equation}\label{OP}
G(x,T,\pi^*)=G^*(x,T) \quad  \text{for\ any }\ x\! \in\! E\!\setminus\! (B_0\cup C).
\end{equation}
 To simplify the optimization problem (\ref{G^*}), we give the  following result revealing that it suffices to seek for optimal policies
in $\Pi_{rm}$.
\begin{pro}\label{pro-2}
\rm{Let $\pi=\{\pi_n:n\geq 0\}\!\in\! \Pi$. Then, there exists a policy $\pi'=\{\psi_n:n\geq 0\}\!\in\! \Pi_{rm}$ such that for each $x \! \in\! E$ and $t \! \in\! \mathbb{R}_T$, $G(x,t,\pi')=G(x,t,\pi)$.
}
\end{pro}
\begin{proof}
It suffices to show that, there exists a policy $\pi'=\{\psi_n:n\geq 0\}\! \in\! \Pi_{rm}$ such that for each $x\! \in\! E$,
\begin{eqnarray}\label{P-pi}
\begin{cases}
P^{\pi}_x(Z_n\in D_1, A_n=a)=P^{\pi}_x(Z_n\in D_1, A_n=a),\ & D_1\! \in\! \mathcal{B}(E),\  a \! \in\! A(y)\\
	P^{\pi'}_x(Z_{n+1}\in D_2)=	P^{\pi}_x(Z_{n+1}\in D_2),\ & D_2\!\in\! \mathcal{B}(E).
\end{cases}
\end{eqnarray}
Indeed, first define $\psi_0:=\pi_0$ and then $\psi_1(a|y):=P_x^{\pi}(A_1=a|Z_1=y)$.
In general, define $
\psi_n(a|y):=P_x^{\pi}(A_n=a | Z_n=y)$ for all $n \geq 2$.
By the method similar to the proof of Theorem 5.5.1 in \cite{M-94}, we can deduce (\ref{P-pi}).
\end{proof}

Since the reach-avoid problem in semi-MDPs is first considered, we present the difference between our problem and other problems in literatures \cite{DM22,IAC-05,KJ-11,LLWY-22} as follows.

\begin{rem}
	{\rm
	\cite{IAC-05} and \cite{KJ-11} considered reach-avoid problems with action-dependent obstacles for continuous dynamic games and differential games respectively, where the precise algorithms for computing the set of reachable states were presented. \cite{LLWY-22} studied reach-avoid problems in nondeterministic systems and gave a numerical method of computing the maximal probabilistic reachable set. This paper considers maximal reach-avoid probability in semi-MDPs with time-varying obstacle sets.
	
	As for reach-avoid probability studied in \cite{DM22}, we can transform the reach-avoid probability into reaching probability $P_x^{\pi}(\tau_C<\infty)$ by assuming the fixed obstacle set and the fixed target set to be closed under any policy. This method can also be applied to the semi-Markov scenario. However, our model involves a sequence of obstacle sets $\{B_n: n \geq 0\}$, and it is impossible to define a new semi-Markov kernel to make $B_n$ closed at $n$'th step.
	Furthermore, by establishing a equivalent model to deal with the problem of distinguishing different situations when transforming at different steps under the stochastic kernel $Q$, we find the equivalent model does not satisfies the ergodic condition, therefore, the long-run average reward method in \cite{DM22} is also not applicable since the method of transforming into the long-run average reward needs the ergodic condition (see \cite{GLL-00,GB-98,GP-01}).}
\end{rem}

From the above argument, we need to present an improved value-type method different from that in \cite{DM22}, to compute the maximal reach-avoid probability defined in (\ref{G^*}) and its $\epsilon$-optimal policy. Therefore, establishing a related model and proving the equivalence of such two reach-avoid probabilities in these two models, presenting some special properties of such model, and giving the improved value-type method for computing the maximal reach-avoid probability of original model (\ref{Model}), consist of the main content of this paper.

\section{Construction of an equivalent model}\label{sec-3}
Since it is difficult to distinguish the situation of transformation at different step  under the stochastic kernel $Q$, we construct another related model to transfer the non-homogeneous model \eqref{Model} into a homogeneous one in this section. For this purpose, let
\begin{eqnarray*}
	N_t:=\max\{n:\sigma_n \leq t\},\quad t \! \in\! \mathbb{R}_T
\end{eqnarray*}
denote the total jump number of $X_t$ on the time interval $[0,t]$ and $Y_t:=(X_t, N_t)$. Obviously, $Y_t$ has the state space $S:=E \times \mathbb{Z}_+$, where $\mathbb{Z}_+:=\{0,1,2,\ldots\}$.
Denote
\begin{eqnarray}\label{eq3.1}
	\tilde{B}_n:=B_n\times \{n\}\ (n\geq 0),\ \ \tilde{B}:=\bigcup_{n=0}^{\infty}\tilde{B}_n.
\end{eqnarray}
Therefore, it is easy to see that (\ref{def-tau}) can be rewritten as
\begin{eqnarray}\label{eq3.2}
\begin{cases}
\tau_{_C}=\inf\{t\geq 0:Y_t\!\in\! C\times \mathbb{Z}_+\}\\
\bar{\tau}=\inf\{t\geq 0:Y_t\!\in\! \tilde{B}\}.
\end{cases}
\end{eqnarray}
Since $\{\tau_{_C}\leq \bar{\tau}\wedge t\}$ does not depend on the evolution after cemetery-hitting time $\bar{\tau}$, we define for all $(x,k)\! \in\! S$,
\begin{eqnarray*}
A(x,k):=\begin{cases}
A(x),\ \ \  & \text{if}\ (x,k)\! \in\! S \setminus\! \tilde{B}\\
\{\Delta^*\},\ & \text{if}\ (x,k)\! \in\! \tilde{B},
\end{cases}
\end{eqnarray*}
where $\Delta^*$ is a special action such that the process remaining at the current state forever.
Moreover, define a new transition kernel as follows: for all $(x,k)\! \in\! S$ and $t\! \in\! \mathbb{R}_T$,
\begin{eqnarray}\label{tilde-Q}
	\begin{cases}		\tilde{Q}((\cdot,n),t|(x,k),a)=Q(\cdot,t|x,a)\delta_{k+1,n},\ & \text{if}\ (x,k)\! \in\! S\! \setminus\! \tilde{B}, a\!\in\! A(x)\\
		\tilde{Q}((\cdot,n),t|(x,k),\Delta^*)=0,\ & \text{if}\ (x,k)\!\in\! \tilde{B}.
	\end{cases}
\end{eqnarray}
\begin{rem}
\rm{It is easy to prove that the above new transition kernel satisfy the assumption in Proposition \ref{As-2}, i.e.,
there exist positive constants $\delta$ and $\epsilon_0$ such that
\begin{eqnarray}\label{As-Q}
	\tilde{Q}(S,\delta|(x,k),a) \leq 1-\epsilon_0\quad \text{for\ all} \ (x,k) \in S \ \text{and}\ a \in A(x,k).
\end{eqnarray}}
\end{rem}
Consider a new semi-MDP model
\begin{equation}\label{newModel}
	\{S=E\!\times\! \mathbb{Z}_+, \tilde{B}, \tilde{C}, (A(x,k)\!\subset\! \tilde{A}: (x,k)\! \in\! S), \tilde{Q}(\cdot,\cdot| (x,k), a)\},
\end{equation}
where $\tilde{B}=\cup_{n=0}^{\infty}B_n\!\times\!\{n\}$, $\tilde{C}=C\! \times\! \mathbb{Z}_+$ and $\tilde{A}=\cup_{(x,k)\in S}A(x,k)$. Regarding to the model (\ref{newModel}), let $\tilde{\Pi}$, $\tilde{\Pi}_{rm}$, $\tilde{\Pi}_s$ and $\tilde{\Pi}_{sd}$ denote the set of all randomized history-dependent policies, the set of all randomized Markov policies and the set of all randomized stationary (Markov) policies and set of all deterministic stationary Markov policies, respectively. Clearly, $\tilde{\Pi}_{sd} \subset \tilde{\Pi}_s \subset \tilde{\Pi}_{rm} \subset \tilde{\Pi}$.

\begin{lem}\label{lem3.1}
{\rm	\begin{description}
		\item[(i)] Suppose that $\pi=\{\psi_n\!: n\geq 0\}\!\in\! \Pi_{rm}$. Define
		\begin{eqnarray}\label{eq3.2a}
			\tilde{\psi}(\cdot|x,n):=
			\begin{cases}
				\psi_n(\cdot|x)\ & \text{if}\ x\! \notin\! B_n,\\
				\delta_{\Delta^*}(\cdot)\ & \text{if}\ x\! \in\! B_n
			\end{cases}\quad \text{for}\ (x,n)\!\in\! S.
		\end{eqnarray}
		Then, $\tilde{\psi}:=\{\tilde{\psi},\tilde{\psi},\cdots\}\!\in\! \tilde{\Pi}_{s}$.
	\item[(ii)] Suppose that $\tilde{\psi}=\{\tilde{\psi},\tilde{\psi},\cdots\}\in \tilde{\Pi}_{s}$. Define
	\begin{eqnarray}\label{eq3.2aa}
		\psi_n(\cdot| x):=
		\begin{cases}
			\tilde{\psi}(\cdot|x,n)\ & \text{if}\ x \notin B_n,\\
			g_n(\cdot|x)\ & \text{if}\ x \in B_n,
		\end{cases}\quad \text{for}\ x\in E,
	\end{eqnarray}
	where $\{g_n(\cdot |x):n \geq 0\}$ is a sequence of probability measures on $A(x)$ for any $x\!\in\! E$.
	Then, $\pi:=\{\psi_n\!:n \geq 0\}\!\in\! \Pi_{rm}$.
	\end{description}}
\end{lem}

\begin{proof}
Obvious.
\end{proof}

Let $\tilde{Y}_t=(\tilde{X}_t,\tilde{N}_t)$ be the semi-Markov process defined by (\ref{newModel}) and $\{\tilde{\sigma}_0=0,\tilde{\sigma}_n:n\geq 1\}$ be the jumping times of $\tilde{Y}_t$. Define
\begin{eqnarray}\label{eq3.6a}
\begin{cases}
\tau_{\tilde{C}}=\inf\{t\geq 0:\tilde{Y}_t\in \tilde{C}\},\\
\tau_{\tilde{B}}=\inf\{t\geq 0:\tilde{Y}_t\in \tilde{B}\}.
\end{cases}
\end{eqnarray}
Since $C$ is assumed to be uniformly-absorbing, we know that $\tilde{B}$ and $\tilde{C}$ are also uniformly-absorbing. Hence, $\tau_{\tilde{C}}\vee \tau_{\tilde{B}}=\infty$ under any policy. For any $\tilde{\psi}\in \tilde{\Pi}_s$ and $(x,k)\in S \setminus \tilde{C}$, define
\begin{eqnarray}\label{eq3.7a}
\tilde{G}(x,k,t,\tilde{\psi}):=P_{(x,k)}^{\tilde{\psi}}(\tau_{\tilde{C}}\leq t),\ t\in \mathbb{R}_T,
\end{eqnarray}
and
\begin{eqnarray}\label{eq3.8a}
\tilde{G}^*(x,k,t):=\sup_{\tilde{\psi}\in \tilde{\Pi}_s}\tilde{G}(x,k,t,\tilde{\psi}),\ (x,k)\in S\! \setminus\! \tilde{C},\ t\in \mathbb{R}_T,
\end{eqnarray}
where $P_{(x,k)}^{\tilde{\psi}}$ denotes the probability measure starting from $(x,k)$ under policy $\tilde{\psi}$.

According to the evolution of model (\ref{newModel}), we give the following definitions.
Let $\mathcal{M}$  be the set of Borel-measurable functions: $W: (S\! \setminus\! \tilde{C})\!
\times\! \mathbb{R}_T \rightarrow [0,1]$ satisfying $W(x,k,t)=0$ for all $x \in B_k$ with $k \geq 0$.
In addition, for any $(x,k,t) \in (S\! \setminus\! \tilde{C})\!\times\! \mathbb{R}_T$, $a \in A(x,k)$ and $\tilde{\psi} \in \tilde{\Pi}_{s}$, we define the operators $\mathcal{L}^a$ and $\mathcal{L}^{\tilde{\psi}}$ on $\mathcal{M}$ as follows:
\begin{eqnarray}\label{operator}\nonumber
\!\!\!\!\!\!&&\mathcal{L}^aW(x,k,t)\!\!:=\! \tilde{Q}((C,k\!\!+\!\!1),t|(x,k),a)+\!\!\int_0^t \!\!\int_{E \setminus (B_{k\!+\!1} \cup C)}\!\!\! \tilde{Q}((dy,k\!\!+\!\!1), du|(x,k),a)W(y,k\!\!+\!\!1,t\!\!-\!\!u)\\
&&\quad\quad\quad\quad=[Q(C,t|x,a)+\!\!\int_0^t\!\!\int_{E \setminus (B_{k \!+\!1} \cup C)}\!\!\!\! Q(dy, du|x,a)W(y,k\!+\!1,t \!-  \! u)]\mathbf{1}_{E \setminus B_k}(x),\\
&&\mathcal{L}^{\tilde{\psi}}W(x,k,t)\!\!:=\!\!\sum_{a \in A(x,k)}\!\!\tilde{\psi}(a|x,k)\mathcal{L}^aW(x,k,t)\nonumber \\
&&\quad\quad\quad\quad=\mathbf{1}_{E \setminus B_k}(x)\!\!\sum_{a \in A(x)}\!\!\pi(a|x)\mathcal{L}^aW(x,k,t)=\mathbf{1}_{E \setminus B_k}(x)\mathcal{L}^{\pi(\cdot|x)}W(x,k,t),
\end{eqnarray}
where $\pi(a|x)=\tilde{\psi}(a|x,k)$ for all $x \in E \!\setminus\! B_k$.

In order to compute $\tilde{G}(x,k,t,\tilde{\psi})$, we also define for $\tilde{\psi}\in \tilde{\Pi}_s$,
\begin{eqnarray}\label{eq3.12aa}
\begin{cases}
\tilde{G}_0(x,k,t,\tilde{\psi}):=0,\\
\tilde{G}_{n}(x,k,t,\tilde{\psi}):=P_{(x,k)}^{\tilde{\psi}}
(\tau_{\tilde{C}}=\tilde{\sigma}_{n}\leq t), \quad n\geq 1\\
\end{cases}\ \ \text{for}\ (x,k,t)\!\in\! (E\!\setminus\! C)\!\times\!\mathbb{Z}_+\!\times\! \mathbb{R}_T.
\end{eqnarray}
It is obvious that $\tilde{G}(x,k,t,\tilde{\psi})=\sum\limits_{n=0}^{\infty}\tilde{G}_{n}
(x,k,t,\tilde{\psi})$ for all $(x,k,t)\!\in\! (E\!\setminus\! C)\!\times\!\mathbb{Z}_+\!\times\! \mathbb{R}_T$.

The following theorem reveals the equivalence of $\tilde{G}(x,0,t,\tilde{\psi})$ and $G(x,t,\pi)$.
\begin{thm}\label{thm3.1}{\rm
\begin{description}
	\item[(i)] Let $\pi=\{\psi_n:n\geq 0\}\in \Pi_{rm}$ and $\tilde{\psi}\in\tilde{\Pi}_s$ be defined by (\ref{eq3.2a}).
	Then,
	\begin{eqnarray}\label{eq3.9a}
\tilde{G}(x,0,t,\tilde{\psi})=G(x,t,\pi),\quad \text{for}\ x\!\in\! E\! \setminus\! C,\ t\!\in\! \mathbb{R}_T.
	\end{eqnarray}
   \item[(ii)] Let $\tilde{\psi}\in \tilde{\Pi}_s$ and $\pi=\{\psi_n:n\geq 0\}$ be defined by (\ref{eq3.2aa}). Then,
   \begin{eqnarray}\label{eq3.9b}
   	G(x,t,\pi)=\tilde{G}(x,0,t,\tilde{\psi}),\quad \text{for}\ x\!\in\! E\! \setminus\! C,\ t\!\in\! \mathbb{R}_T.
   \end{eqnarray}
   Moreover,
   \begin{eqnarray}\label{eq3.10a}
   	G^*(x,t)=\tilde{G}^*(x,0,t),\quad \text{for}\ x\!\in\! E \!\setminus\! C,\ t\!\in\! \mathbb{R}_T.
   \end{eqnarray}
\end{description}}
\end{thm}
\begin{proof}
For any $x\! \in\! E\! \setminus\! C$, denote
\begin{eqnarray*}
	\begin{cases}
	F_n(x,D,t,\pi):=P_{x}^{\pi}(X_{\sigma_n}\! \in\! D,\ \sigma_n \leq \bar{\tau} \wedge t), \quad D\! \subset\! E\! \setminus\! (B_n\! \cup\! C), \ n \geq 0\\
	\tilde{F}_n(x,D,t, \tilde{\psi}):=P_{(x,0)}^{\tilde{\psi}}(\tilde{Y}_{\tilde{\sigma}_n}\! \in\! D\! \times\! \{n\},\ \tilde{\sigma}_n \leq t), \quad D\! \subset\! E\! \setminus\! (B_n\! \cup\! C), \ n \geq 0.
	\end{cases}
\end{eqnarray*}
We first prove that for any $(x,t)\!\in\! (E\!\setminus\! C)\!\times\! \mathbb{R}_T$,
\begin{eqnarray}\label{F_n}
	\tilde{F}_n(x,D,t,\tilde{\psi})=F_n(x,D,t,\pi), \quad D\! \subset\! E\! \setminus\! (B_n\! \cup\! C), \ n \geq 0.
\end{eqnarray}
Indeed, for any $(x,t)\!\in\! (E\!\setminus\! C)\times\! \mathbb{R}_T$, $\tilde{F}_0(x,D,t,\tilde{\psi})=\mathbf{1}_D(x)=F_0(x,D,t, \pi)$,
and noting that $C$ is uniformly-absorbing, we have
\begin{eqnarray*}
	\tilde{F}_1(x, D, t, \tilde{\psi})&=&\sum\limits_{a \in A(x,0)}\tilde{\psi}(a|x,0)\tilde{Q}((D,1),t|(x,0),a)\\
	&=&\sum\limits_{a \in A(x)}\psi_0(a|x) Q(D,t|x,a)\mathbf{1}_{\{E \setminus (B_0 \cup C)\}}(x)=F_1(x, D,t, \pi).
\end{eqnarray*}
Suppose that (\ref{F_n}) holds for some $n\geq 0$. Then,
\begin{eqnarray*}
&&\tilde{F}_{n+1}(x,D,t,\tilde{\psi})\\
&=&E_{(x,0)}^{\tilde{\psi}}[\mathbf{1}_{\{\tilde{Y}_{\tilde{\sigma}_{n+1}}
\in D\times \{n+1\},\tilde{\sigma}_{n+1}\leq t\}}]\\	&=&E_{(x,0)}^{\tilde{\psi}}[\mathbf{1}_{\{\tilde{Y}_{\tilde{\sigma}_{n}}\in (E\setminus (B_n\cup C))\times \{n\},\tilde{\sigma}_n\leq t\}}\mathbf{1}_{\{\tilde{Y}_{\tilde{\sigma}_{n+1}}\in D\times \{n+1\},\tilde{\sigma}_{n+1}\leq t\}}]\\
&=&\int_{E\setminus (B_n\cup C)}\int_0^t\tilde{F}_n(x,dy,du,\tilde{\psi})E_{(y,n)}
	^{\tilde{\psi}}[\mathbf{1}_{\{\tilde{Y}_{\tilde{\sigma}_{1}}\in D\times \{1\},\tilde{\sigma}_1\leq (t-u)\}}]\\
	&=&\int_{E\setminus (B_n\cup C)}\int_0^t\tilde{F}_n(x,dy,du,\tilde{\psi})\sum_{a\in A(y,n)}\tilde{\psi}(a|y,n)\tilde{Q}((D,n\!+\!1),t\!-\!u|(y,n),a)\\
	&=&\int_{E\setminus (B_n\cup C)}\int_0^t F_n(x,dy,du,\pi)\sum_{a\in A(y)}\psi_n(a|y)Q(D,t\!-\!u|y,a)\\
	&=&\int_{E\setminus (B_n\cup C)}\int_0^tF_n(x,dy,du,\pi)E_{y}
	^{\pi}[\mathbf{1}_{\{X_{\sigma_{1}}\in D,\sigma_1\leq \bar{\tau}\wedge(t-u)\}}]\\
	&=&E_x^{\pi}[\mathbf{1}_{\{X_{\sigma_{n}}\in E\setminus (B_n\cup C),\sigma_n\leq \bar{\tau}\wedge t\}}\mathbf{1}_{\{X_{\sigma_{n+1}}\in D,\sigma_{n+1}\leq \bar{\tau}\wedge t\}}]=F_{n+1}(x,D,t,\pi),
\end{eqnarray*}
where $E_{(x,k)}^{\tilde{\psi}}$ denotes the mathematical expectation under $P_{(x,k)}^{\tilde{\psi}}$. Thus, (\ref{F_n}) holds for all $n\geq 0$.
Now prove (\ref{eq3.9a}).
By the above argument, we have obtained that
\begin{eqnarray*}
	\tilde{G}_1(x,0,t,\tilde{\psi})=P_x^{\pi}(\tau_C=\sigma_1\leq \bar{\tau}\wedge t) \quad \text{for\ all \ } x\!\in\! E\!\setminus\! C.
\end{eqnarray*}
Furthermore, for any $n \geq 1$,
\begin{eqnarray*}
\tilde{G}_{n+1}(x,0,t,\tilde{\psi})
&=&P_{(x,0)}^{\tilde{\psi}}(\tau_{\tilde{C}}=\tilde{\sigma}_{n+1}\leq t)\\
&=&\int_{E\setminus (B_n\cup C)}\int_0^t\tilde{F}_n(x,dy,du,\tilde{\psi})\sum_{a\in A(y,n)}\tilde{\psi}(a|y,n)\tilde{Q}((C,n\!+\!1),t\!-\!u|(y,n),a)\\
	&=&\int_{E\setminus (B_n\cup C)}\int_0^tF_n(x,dy,du,\pi)\sum_{a\in A(y)}\psi_n(a|y)Q(C,t\!-\!u|y,a)\\
	&=&P_x^{\pi}(\tau_C=\sigma_{n+1}\leq \bar{\tau}\wedge t)=G_{n+1}(x,t,\pi).
\end{eqnarray*}
Summing over $n$, yields (\ref{eq3.9a}). (\ref{eq3.9b}) can be similarly proved.
By (i) and (ii), taking supremum over $\pi\in\Pi_{rd}$ and $\tilde{\psi}\in\tilde{\Pi}_s$, we get that (\ref{eq3.10a}) holds.
\end{proof}
By Theorem~\ref{thm3.1} and Lemma~\ref{lem3.1}, we only need to find $\tilde{\psi}^*\in \tilde{\Pi}_s$ such that
\begin{eqnarray*}
	\tilde{G}(x,0,T,\tilde{\psi}^*)=\tilde{G}^*(x,0,T) \quad \text{for}\ x\!\in\! E\!\setminus\! C.
\end{eqnarray*}

\section{Analysis of the model (\ref{newModel})}\label{sec-4}
In this section, we use the regular method to prove the existence of an optimal policy, and then illustrate several useful properties of the model (\ref{newModel}). Then, in model (\ref{newModel}), we compute $\tilde{G}^*(x,0,T)$ for every $x \in E \setminus C$ (i.e., steps 1-3 in Algorithm \ref{al-finite}). By using Theorem \ref{thm3.1} and Lemma \ref{lem3.1}, we transform $\tilde{G}^*(x,0,T)$ and its $\epsilon$-optimal policy into the maximal reach-avoid probability and its $\epsilon$-optimal policy in the original model (\ref{Model}) at step 4 in Algorithm \ref{al-finite}.

\subsection{The existence of an optimal policy}

In this subsection, we mainly present the existence of an optimal policy so that we can give the improved value-type algorithm of the maximal reach-avoid probability and its optimal policy on its basis.

First, we give the following proposition, which is similar with Lemma 3.3 in \cite{HGW-23}. For convenience of later citation, we give a simple proof here.

\begin{pro}\label{th-unique}
\rm{Suppose that (\ref{As-Q}) holds. Let $\tilde{\psi}\in \tilde{\Pi}_s$. For any $H \in \mathcal{M}$ and $(x,k,t)\in (E\setminus C)\times \mathbb{Z}_+\times \mathbb{R}_T$, we have
\begin{description}
\item [(a)] If $H(x,k,t) \leq \mathcal{L}^{\tilde{\psi}}H(x,k,t)$, then $H(x,k,t) \leq \tilde{G}(x,k,t,\tilde{\psi})$;
			
\item [(b)] If $H(x,k,t) \geq \mathcal{L}^{\tilde{\psi}}H(x,k,t)$, then $H(x,k,t) \geq \tilde{G}(x,k,t,\tilde{\psi})$;
			
\item [(c)] $\tilde{G}(\cdot,\cdot, \cdot,\tilde{\psi})$ is the unique solution to the equation $W=\mathcal{L}^{\tilde{\psi}}W$ on $\mathcal{M}$.
\end{description}
}
\end{pro}

\begin{proof}
First prove (a). It is easy to check that for any $\tilde{\psi}\in \tilde{\Pi}_s$,
\begin{eqnarray*}
\tilde{G}(x,k,t,\tilde{\psi})=\tilde{\mathcal{L}}^{\tilde{\psi}}
\tilde{G}(x,k,t,\tilde{\psi}),\ \ (x,k,t)\in (E\setminus C)\times \mathbb{Z}_+\!\times \mathbb{R}_T.
\end{eqnarray*}

Denote $J(x,k,t):=H(x,k,t)-\tilde{G}(x,k,t,\tilde{\psi})$. Then, $
J(x,k,t) \leq \tilde{\mathcal{L}}^{\tilde{\psi}}J(x,k,t)$,
where $\tilde{\mathcal{L}}^{\tilde{\psi}}J(x,k,t)=\!\!\!\!\sum\limits_{a \in A(x,k)}\!\!\tilde{\psi}(a|x,k)\int_0^t\int_{E \setminus (B_{k\!+\!1} \cup C)}\!\! \tilde{Q}((dy,k\!+\!1), du|(x,k),a)J(y,k\!+\!1,t-u)$.
Take $\delta$ and $\epsilon_0$ as in Proposition \ref{As-2} and define $F_{\delta}(t):=(1-\epsilon_0)\mathbf{1}_{[0,\delta)}(t)
+\mathbf{1}_{(\delta,\infty)}(t)$. By induction argument, we can see that for all $n\geq 0$,
$J(x,k,t)\leq(\tilde{\mathcal{L}}^{\tilde{\psi}}\cdots \tilde{\mathcal{L}}^{\tilde{\psi}})J(x,k,t)\leq F_{\delta}^{*(n)}(t)$,
where $F_{\delta}^{*(n)}(t)$ denote the $n$-fold convolution of $F_{\delta}(t)$. However, by Theorem 1 in \cite{J86}, we have
$F_{\delta}^{*(n)}(t) \leq (1-\epsilon_0^{\tilde{K}})^{\lfloor\frac{n}{\tilde{K}}\rfloor}\ (n>\tilde{K})$,
where $\tilde{K}$ is an integer satisfying $\tilde{K}>\frac{T}{\delta}$, and
$\lfloor r\rfloor$ is the largest integer not bigger than $r$. Therefore, $J(x,k,t) \leq (1-\epsilon_0^{\tilde{K}})^{\lfloor\frac{n}{\tilde{K}}\rfloor}\ (n>\tilde{K})$, which, implying that $H(x,k,t)\leq G(x,k,t,\tilde{\pi})$.
A similar argument as in (a) achieves (b). Combining (a) and (b) yield (c).
\end{proof}

Recall that our main aim is to find a policy $\tilde{\psi}^*\in\tilde{\Pi}_s$ such that
\begin{eqnarray*}
	\tilde{G}(x,k,T,\tilde{\psi}^*)=\tilde{G}^*(x,k,T)=\sup_{\tilde{\psi}\in \tilde{\Pi}_s}G(x,k,T,\tilde{\psi})\quad \text{for\ all} \ (x,k)\in (E\! \setminus\! C)\!\times\! \mathbb{Z}_+.
\end{eqnarray*}
The following theorem presents the existence of an optimal policy $\tilde{\psi}^*$ in model (\ref{newModel}), which is deterministic stationary (i.e., $\tilde{\psi}^*=\tilde{f}^* \in \tilde{\Pi}_{sd}$). Such theorem ensures that Algorithm \ref{al-finite} is meaningful, and thus we put it here as an important result in this subsection.

\begin{thm}\label{OE-optimal}
{\rm Suppose that Assumption \ref{As-1} holds. Then,
\begin{description}
\item[(i)] $(\tilde{G}^*(x,k,t): (x,k,t) \in (E\! \setminus\! C)\! \times\! \mathbb{Z}_+\! \times\! \mathbb{R}_T )$ satisfies the optimality equation (OE):
\begin{eqnarray}\label{OE}
\begin{cases}
W(x,k,t)=\max\limits_{a\in A(x,k)}\mathcal{L}^{a}W(x,k,t),\ &\text{if} \ (x,k,t) \in  (E\! \setminus\! C)\! \times\! \mathbb{Z}_+\! \times\! \mathbb{R}_T,\\
					0 \leq W(x,k,t) \leq 1,\ & \text{if} \ (x,k,t) \in (E\! \setminus\! C)\!\times\! \mathbb{Z}_+\! \times\! \mathbb{R}_T.
\end{cases}
\end{eqnarray}
\item [(ii)] There exists a deterministic stationary policy $\tilde{f}^*\in \tilde{\Pi}_{sd}$ (maybe depend on $T$) such that
\begin{eqnarray*}		\tilde{G}(x,k,T,\tilde{f}^*)=\tilde{G}^*(x,k,T)
\end{eqnarray*}
for all $(x,k)\in (E\! \setminus\! C)\! \times\! \mathbb{Z}_+$. Hence, $\tilde{f}^*$ is optimal for (\ref{newModel}).
			
\item[(iii)] There exists $\pi^*:=\{f^*_n:n\geq 0\} \in \Pi_d$ (maybe depend on $T$) such that
\begin{eqnarray*}
G(x,T,\pi^*)=G^*(x,T)
\end{eqnarray*}
for all $x\in E\! \setminus\! (B_0 \cup C)$. Hence, $\pi^*$ is optimal for (\ref{Model}).
\end{description}
}
\end{thm}

\begin{proof}
First prove (i)-(ii). For all $\tilde{\psi} \in \tilde{\Pi}_s$ and $(x,k) \in (E \setminus C) \times \mathbb{Z}_+$,
	\begin{eqnarray*}
		\tilde{G}(x,k,T,\tilde{\psi})\!\!\!\!&=&\!\!\!\sum_{a \in A(x,k)}\!\!\!\!\tilde{\psi}(a|x,k)[\int_C \!\! \tilde{Q}((dy,k+1),t|(x,k), a)\\
		&&+\!\!\int_0^T \!\! \int_{E \setminus (B_{k+1} \cup C)}\!\!\!\!\!\!\!\!\!\!\tilde{Q}((dy,k+1),du|(x,k),a)\tilde{G}
(y,k+1,T-u,\tilde{\psi})]\\
		&\leq& \max_{a \in A(x,k)}[\int_C \!\! \tilde{Q}((dy,k+1),t|(x,k), a)\\
		&&+\!\!\int_0^T \!\! \int_{E \setminus (B_{k+1} \cup C)}\!\!\!\!\!\!\!\!\!\!\tilde{Q}((dy,k+1),du|(x,k),a)\tilde{G}^*(y,k+1,T-u)]\\
		&=&\max_{a \in A(x,k)}\mathcal{L}^a\tilde{G}^*(x,k,T),
	\end{eqnarray*}
where the last equality is due to the definition of $\tilde{G}^*$. Then, after taking the maximum over $\tilde{\psi} \in \tilde{\Pi}_s$ on the both sides, together with the finiteness of $A(x,k)$ for all $(x,k) \in (E \setminus C) \times \mathbb{Z}_+$, there exists $\tilde{f}^* \in \tilde{\Pi}_{sd}$ such that
\begin{eqnarray}\label{G-leq}
	\tilde{G}^*(x,k,T) \leq \max_{a \in A(x,k)}\mathcal{L}^a\tilde{G}^*(x,k,T)=\mathcal{L}^{\tilde{f}^*}\tilde{G}^*(x,k,T).
\end{eqnarray}
Moreover, by $\tilde{\Pi}_{sd} \subset \tilde{\Pi}_s$ and Proposition \ref{th-unique}(a), we have $\tilde{G}^*(x,k,T) \leq \tilde{G}(x,k,T,\tilde{f}^* )$, which forces that $\tilde{G}^*(x,k,T)=\tilde{G}(x,k,T,\tilde{f}^* )$ since
$\tilde{G}(x,k,T, \tilde{f}^*) \leq \tilde{G}^*(x,k,T)$ is obvious. Therefore, (ii) is proved. (i) follows from  $\tilde{G}^*(x,k,T)=\tilde{G}(x,k,T,\tilde{f}^* )$ and (\ref{G-leq}).

Next prove (iii). Define $\pi^* =\{f^*_n:n\geq 0\}$ as below. For every $x \in E$,
	\begin{eqnarray*}
		f^*_n(x):=\begin{cases}
			\tilde{f}^*(x,n),\quad & \text{if}\ x \notin B_n\\
			g_n(x),\quad & \text{if}\ x \in B_n,
		\end{cases}
	\end{eqnarray*}
	where $\{g_n(x): n \geq 0\}$ is a sequence of actions in $A(x)$ for any $x \in E$, and $\tilde{f}^*$ is given by (ii). By Theorem \ref{thm3.1}(ii), $\pi^* \in \Pi_d$ is an optimal policy for (\ref{Model}).
\end{proof}

To end this subsection, we present several useful properties of the model (\ref{newModel}) as below.

\begin{thm}\label{thm useful}
\rm{ For the model \eqref{newModel}, the following assertions hold.
\begin{description}
\item[(i)] If $B_k \subset B_{k-1}$ for all $k \geq 1$, then, for all $(x,k)\! \in\! S\! \setminus\! \tilde{C}$,
\begin{eqnarray}\label{tilde-G}
\tilde{G}(x,k-1,t,\tilde{\psi}) \leq \tilde{G}(x,k,t,\tilde{\psi}),\ t\! \in\! \mathbb{R}_T\ \text{and} \ \tilde{\psi}\!\in\! \tilde{\Pi}_s,
\end{eqnarray}
and thus
\begin{eqnarray}\label{G*}
\tilde{G}^*(x,k-1,T) \leq \tilde{G}^*(x,k,T).
\end{eqnarray}
			
\item[(ii)] If $B_{k-1} \subset B_k$ for all $k \geq 1$, then, for all $(x,k)\! \in\! S\! \setminus\! \tilde{C}$,
\begin{eqnarray}\label{tilde-G-1}
\tilde{G}(x,k-1,t,\tilde{\psi}) \geq \tilde{G}(x,k,t,\tilde{\psi})\quad \text{for \ all \ } t\! \in\! \mathbb{R}_T\ \text{and} \ \tilde{\psi}\!\in\! \tilde{\Pi}_s,
\end{eqnarray}
and thus
\begin{eqnarray}\label{G*-1}
\tilde{G}^*(x,k-1,T) \geq \tilde{G}^*(x,k,T).
\end{eqnarray}
\item[(iii)] Under the condition of (i) or (ii), if $\lim\limits_{k \rightarrow \infty}B_k=D\ ( \subsetneq E\! \setminus\! C)$, then for every $(x,k)\! \in\! S\! \setminus\! \tilde{C}$, $\tilde{\psi} \in \tilde{\Pi}_s$ and $t \in \mathbb{R}_T$, $\lim\limits_{k \rightarrow \infty}\tilde{G}(x,k,t,\tilde{\psi})$ is the hitting probability to $\tilde{C}$ from state $(x,k)$ with a fixed obstacle set $D$ under policy $\tilde{\psi}$ within $[0,t]$.
\end{description}}
\end{thm}
\begin{proof}
\rm{ (i) Consider another equivalent model as below:
\begin{eqnarray}\label{newnewModel}
\{S=E\!\times\! \mathbb{Z}_{+}, \bar{B}, \tilde{C}, (\bar{A}(x,k)\subset \bar{A}: (x,k) \in S), \bar{Q}(\cdot,\cdot| (x,k), a)\},
\end{eqnarray}
where $\bar{B}=\cup_{n=0}^{\infty}\bar{B}_n\times\{n\}$ with $\bar{B}_n=B_{n+1}$, $\bar{A}=\cup_{(x,k)\in S}\bar{A}(x,k)$ with $\bar{A}(x,k)=A(x,k+1)$ and $\bar{Q}(\cdot,\cdot|(x,k),a)=\tilde{Q}(\cdot,\cdot|(x,k+1),a)$.
		
Let $\bar{Y}_t=(\bar{X}_t,\bar{N}_t)$ be the process determined by (\ref{newnewModel}). For $\tilde{\psi}\in\tilde{\Pi}_s$, define $\bar{\psi}$ by $\bar{\psi}(\cdot |x,k)=\tilde{\psi}(\cdot |x,k+1)$.
It is easy to see that the evolution of $\bar{Y}_t$ under $P^{\bar{\psi}}_{(x,k)}$ is same as the evolution of $\tilde{Y}_t$ under $P^{\tilde{\psi}}_{(x,k+1)}$. Therefore, noting $\bar{B}_n=B_{n+1}$, we have
\begin{eqnarray}\label{G-B}
P^{\bar{\psi}}_{(x,k)}(\bar{\tau}_{\tilde{C}}\leq \bar{\tau}_{\bar{B}}\wedge t)
=P^{\tilde{\psi}}_{(x,k+1)}(\tilde{\tau}_{\tilde{C}}\leq \tilde{\tau}_{\tilde{B}}\wedge t)=\tilde{G}(x,k+1,t,\tilde{\psi}),
\end{eqnarray}
where $\bar{\tau}_{\tilde{C}}=\inf\{t\geq 0:\bar{Y}_t\in \tilde{C}\}$ and $\bar{\tau}_{\bar{B}}=\inf\{t\geq 0:\bar{Y}_t\in \bar{B}\}$.
However, since $\bar{B}\subset \tilde{B}$ (from $B_k \subset B_{k-1}$ for all $k \geq 1$), we see that if $\bar{Y}_0=(x,k)\in S\!\setminus\! \tilde{B}$, then $\bar{\tau}_{\bar{B}}\geq \bar{\tau}_{\tilde{B}}$, and hence by (\ref{G-B}) and $\tilde{Q}(\cdot,\cdot|(x,k),a)=\bar{Q}(\cdot,\cdot|(x,k+1),a)=Q(\cdot,\cdot |x,a)$ for $(x,k)\in S\!\setminus\! \tilde{B}$, we have
\begin{eqnarray*}	\tilde{G}(x,k,t,\tilde{\psi})&=&P^{\tilde{\psi}}_{(x,k)}(\tilde{\tau}_
{\tilde{C}}\leq \tilde{\tau}_{\tilde{B}}\wedge t)=P^{\bar{\psi}}_{(x,k)}(\bar{\tau}_{\tilde{C}}\leq \bar{\tau}_{\tilde{B}}\wedge t)\\
&\leq& P^{\bar{\psi}}_{(x,k)}(\bar{\tau}_{\tilde{C}}\leq \bar{\tau}_{\bar{B}}\wedge t)\ \ (\bar{B}\ \text{replaced\ with}\ \tilde{B}\ \text{for}\ \bar{Y}_t)\\
&=&\tilde{G}(x,k+1,t,\tilde{\psi}).
\end{eqnarray*}
Therefore, by the arbitrary of $\tilde{\psi}$ and $t \in \mathbb{R}_T$, (\ref{G*}) holds.
		
(ii)\ We modified the equivalent model  (\ref{newnewModel}) as below:
\begin{eqnarray}\label{newnewModel-1}
\{S=E\!\times\! \mathbb{Z}_{+}, \hat{B}, \tilde{C}, (\hat{A}(x,k)\subset \hat{A}: (x,k) \in S), \hat{Q}(\cdot,\cdot| (x,k), a)\},
\end{eqnarray}
where $\hat{B}=\cup_{n=0}^{\infty}\hat{B}_n\times\{n\}$ with $\hat{B}_0=\emptyset, \hat{B}_n=B_{n-1}\ (n\geq 1)$, $\hat{A}=\cup_{(x,k)\in S}\hat{A}(x,k)$ with $\hat{A}(x,0)=A(x)$, $\hat{A}(x,k)=A(x,k-1)\ (k\geq 1)$ and $\hat{Q}(\cdot,\cdot|(x,0),a)=Q(\cdot,\cdot|x,a)$, $\hat{Q}(\cdot,\cdot|(x,k),a)=\tilde{Q}(\cdot,\cdot|(x,k-1),a)\ (k\geq 1)$.
		
Let $\hat{Y}_t=(\hat{X}_t,\hat{N}_t)$ be the process determined by (\ref{newnewModel-1}). For $\tilde{\psi}$ given by (\ref{eq3.2a}), define $\hat{\psi}$ by $\hat{\psi}(\cdot |x,0)=\psi(\cdot|x)$ and $\hat{\psi}(\cdot |x,k)=\tilde{\psi}(\cdot |x,k-1)\ (k\geq 1)$.
It is easy to see that the evolution of $\hat{Y}_t$ under $P^{\hat{\psi}}_{(x,k+1)}$ is same as the evolution of $\tilde{Y}_t$ under $P^{\tilde{\psi}}_{(x,k)}$. Therefore, noting $\hat{B}_n=B_{n-1}$, we have
\begin{eqnarray}\label{G-B1}
P^{\hat{\psi}}_{(x,k+1)}(\hat{\tau}_{\tilde{C}}\leq \hat{\tau}_{\hat{B}}\wedge t)
=P^{\tilde{\psi}}_{(x,k)}(\tilde{\tau}_{\tilde{C}}\leq \tilde{\tau}_{\tilde{B}}\wedge t)=\tilde{G}(x,k,t,\tilde{\psi}),
\end{eqnarray}
where $\hat{\tau}_{\tilde{C}}=\inf\{t\geq 0:\hat{Y}_t\in \tilde{C}\}$ and $\hat{\tau}_{\hat{B}}=\inf\{t\geq 0:\hat{Y}_t\in \hat{B}\}$. However, since $\hat{B}\subset\tilde{B}$ (from $B_{k-1} \subset B_k$ for all $k \geq 1$), if $\hat{Y}_0=(x,k+1)\in S\!\setminus\! \tilde{B}$, then $\hat{\tau}_{\tilde{B}}\leq \hat{\tau}_{\hat{B}}$, and hence by (\ref{G-B1}) and $\hat{Q}(\cdot,\cdot|(x,k+1),a)=\tilde{Q}(\cdot,\cdot|(x,k),a)=Q(\cdot,\cdot |x,a)$ for $(x,k+1)\in S\!\setminus\! \tilde{B}$, we have
\begin{eqnarray*}
\tilde{G}(x,k,t,\tilde{\pi})
&=&P^{\hat{\psi}}_{(x,k+1)}(\hat{\tau}_{\tilde{C}}\leq \hat{\tau}_{\hat{B}}\wedge t)\\
&\geq & P^{\hat{\psi}}_{(x,k+1)}(\hat{\tau}_{\tilde{C}}\leq \hat{\tau}_{\tilde{B}}\wedge t)\ \ (\hat{B}\  \text{replaced\ with}\ \tilde{B}\ \text{for}\ \hat{Y}_t)\\
&=& P^{\tilde{\psi}}_{(x,k+1)}(\tilde{\tau}_{\tilde{C}}\leq \tilde{\tau}_{\tilde{B}}\wedge t)=\tilde{G}(x,k+1,t,\tilde{\psi}).
\end{eqnarray*}
Therefore, by the arbitrary of $\tilde{\psi}$ and $t \in \mathbb{R}_T$, (\ref{G*-1}) holds.
		
(iii) Obviously, when $E$ is finite, from (\ref{tilde-G}), $\tilde{G}(x,k,t,\tilde{\psi})$ has no relationship with the obstacles $B_0,\ B_1, \cdots,\ B_{k-1}$, and thus by $\lim\limits_{k \rightarrow \infty}B_k=D\ ( \subsetneq E\! \setminus\! C)$, we know that there exists $n_0\geq 0$ such that $B_k=D\ (k\geq n_0)$. Hence, for any $x\in E\setminus (D\cup C)$ and $k\geq n_0$,
$\tilde{G}(x,k,t,\tilde{\psi})=\tilde{G}(x,n_0,t,\tilde{\psi})$,
which is the hitting probability to $\tilde{C}$ from state $x$ with a fixed obstacle set $D$ under policy $\tilde{\pi}$ within $[0,t]$. When $E$ is a Borel set, if $\lim\limits_{k \rightarrow \infty}B_k=D\ ( \subsetneq E\! \setminus\! C)$, then by (\ref{tilde-Q}), we have
\begin{eqnarray*}
&&\lim\limits_{k \rightarrow \infty}|\tilde{Q}((B_{k+1}, k+1), t| (x,k), a)-\tilde{Q}((D, k+1), t| (x,k), a)|\\
&=&\lim\limits_{k \rightarrow \infty}|Q(B_{k+1}, t|x,a)-Q(D,t|x,a)|=0
\end{eqnarray*}
for all $x \in E\!\setminus\! (D\cup C)$, $t \in \mathbb{R}_T$ and $a \in A(x)$.
Then, by (\ref{eq3.7a}), we know that for all $x\in E \!\setminus\! (D\cup C)$, $\lim\limits_{k \rightarrow \infty}\tilde{G}(x,k,t,\tilde{\psi})$ is the hitting probability to $\tilde{C}$ from state $x$ with a fixed obstacle set $D$ under policy $\tilde{\pi}$ within $[0,t]$.}
\end{proof}

\begin{rem}\label{rem-2}
\rm{	By Theorem \ref{thm useful}, when the obstacle set $B_k$ has monotonicity respect to $k$, $\tilde{G}^*$ also has the monotonicity respect to $k$. Therefore, if $B_k \subset B_{k-1}$, then when the process is in a state where it is transferred to neither $C$ nor $B_k$, the probability of hitting the target $C$ is greater than the probability of hitting the target $C$ at the initial time of $x$. Similar property holds in the case that $B_{k-1} \subset B_k$.
}
\end{rem}

\subsection{Improved value-type algorithm}

In this subsection, we mainly present an improved value-type algorithm of computing the maximal reach-avoid probability and its $\epsilon$-optimal policy.

We now define that for $(x,k,t)\in (E \setminus C) \times \mathbb{Z}_+ \times \mathbb{R}_T$,
\begin{eqnarray}\label{W-2}
\begin{cases}
W_1(x,k,t):=\max\limits_{a\in A(x,k)}\tilde{Q}((C,k+1),t|(x,k),a)\\
W_{n+1}(x,k,t):=\max\limits_{a\in A(x,k)}\mathcal{L}^{a}W_n(x,k,t),\quad n \geq 1,
\end{cases}
\end{eqnarray}
and present the characteristic of the above sequence $\{W_n(x,k,t): n \geq 1\}$, which is significant for analyzing  $\tilde{G}^*(x,k,t)$.

\begin{thm}\label{thm-W}
\rm{Suppose that (\ref{As-Q}) holds. Let $\{W_n(x,k,t): n \geq 1\}$ be defined in (\ref{W-2}). Then, we have the following assertions. For all $(x,k,t) \in (E \setminus C) \times \mathbb{Z}_+  \times \mathbb{R}_T $,
\begin{description}
\item[(i)]  $W_n(x,k,t)$ is nondecreasing on $t \in \mathbb{R}_T$ for all $n \geq 1$;

\item[(ii)]  $W_n(x,k,t)\leq W_{n+1}(x,k,t)$ for all $n \geq 1$;

\item[(iii)] $\lim\limits_{n \rightarrow \infty}W_n(x,k,t)=\tilde{G}^*(x,k,t)$.
\end{description}
}
\end{thm}
\begin{proof}
By the definition of $W_n(\cdot,\cdot,\cdot)$, (i) is obvious. As for (ii), it is easy to get that $W_1(\cdot,\cdot,\cdot) \leq W_2(\cdot,\cdot,\cdot)$, and by mathematical induction, we have $W_{n}(\cdot,\cdot,\cdot) \leq W_{n+1}(\cdot,\cdot,\cdot)$ for all $n \geq 1$. Finally we prove (iii). Obviously, $W_n(x,k,t) \in [0,1]$ for all $(x,k,t) \in (E \setminus C) \times \mathbb{Z}_+ \times \mathbb{R}_T$. It follows from the monotone convergence theorem and (ii), that $W^*(x,k,t):=\lim\limits_{n \rightarrow \infty}W_n(x,k,t)$ exists for every $(x,k,t) \in (E \setminus C) \times \mathbb{Z}_+ \times \mathbb{R}_T$. By the finiteness of $A(x,k)$, there exists an action $a^{*(n)}_{x,k} \in A(x,k)\ (n \geq 0)$ such that $\mathcal{L}^{a^{*(n)}_{x,k}}W_n(x,k,t)=\max\limits_{a \in A(x,k)}\mathcal{L}^aW_n(x,k,t)$.
Since $A(x,k)$ is finite and $a^{*(n)}_{x,k}\in A(x,k)\ (n \geq 0)$, there exists an action $a^{*}_{x,k}$ and a sub-sequence $\{n_l: l \geq 0\}$ such that $a^{*(n_l)}_{x,k}=a^*_{x,k}$. Hence, $\mathcal{L}^{a^{*}_{x,k}}W_{n_l}(x,k,t)=\max\limits_{a \in A(x,k)}\mathcal{L}^aW_{n_l}(x,k,t)$ for all $l \geq 0$.
Moreover, we easily get that for all $(x,k) \in (E \setminus C) \times \mathbb{Z}_+$, $\lim\limits_{l \rightarrow \infty}\mathcal{L}^{a^{*}_{x,k}}W_{n_l}(x,k,t)=\mathcal{L}^{a^*_{x,k}}W^*(x,k,t)$. Then, take $\hat{f} \in  \tilde{\Pi}_{sd}$ such that $\hat{f}(x,k)=a^*_{x,k}$ for all $(x,k) \in (E \setminus C) \times \mathbb{Z}_+$. Then, for all $(x,k) \in (E\! \setminus\! C)\! \times \!\mathbb{Z}_+$, $\mathcal{L}^{\hat{f}}W^*(x,k,t)=\mathcal{L}^{a^*_{x,k}}W^*(x,k,t)=\max\limits_{a \in A(x,k)}\mathcal{L}^{a}W^*(x,k,t)$.
Here we have used the fact that $\lim\limits_{l\to \infty}\max\limits_{a \in A(x,k)}\mathcal{L}^{a}W_{n_l}(x,k,t)=\max\limits_{a \in A(x,k)}\mathcal{L}^{a}W^*(x,k,t)$. It follows from (\ref{W-2}) that $
W^*(x,k,t)=\mathcal{L}^{\hat{f}}W^*(x,k,t)$. By Proposition~\ref{th-unique}(c), we know that $\tilde{G}(x,k,t,\hat{f})=\mathcal{L}^{\hat{f}}\tilde{G}(x,k,t,\hat{f})$.
Hence, by Proposition~\ref{th-unique}, we have $W^*(x,k,t)=\tilde{G}(x,k,t,\hat{f})$.

On the other hand, we can prove that for any $\tilde{f}\in \tilde{\Pi}_{sd}$,
\begin{eqnarray}\label{eq4.13}
P^{\tilde{f}}_{(x,k)}(\tau_{\tilde{C}}\leq\tilde{\sigma}_{n}\wedge t)\leq W_n(x,k,t),\ \ n\geq 1.
\end{eqnarray}
Indeed,
$P^{\tilde{f}}_{(x,k)}(\tau_{\tilde{C}}\leq\tilde{\sigma}_{1}\wedge t)=\tilde{Q}((C,k\!\!+\!\!1),t|(x,k),\tilde{f}(x,k))\leq W_1(x,k,t)
$ and for $n\geq 1$, by Markov property and mathematical induction,
\begin{eqnarray*}
&& P^{\tilde{f}}_{(x,k)}(\tau_{\tilde{C}}\leq\tilde{\sigma}_{n+1}\wedge t)\\
&=&\tilde{Q}((C,k\!\!+\!\!1),t|(x,k),\tilde{f}(x,k))\\
&&+\int_0^t\int_{E\setminus (B_{k\!+\!1}\cup C)}\!\!\tilde{Q}((dy,k\!+\!1),du|(x,k),\tilde{f}(x,k))
P^{\tilde{f}}_{(y,k\!+\!1)}(\tau_{\tilde{C}}\leq\tilde{\sigma}_{n}\wedge (t-u))\\
&\leq&\tilde{Q}((C,k\!\!+\!\!1),t|(x,k),\tilde{f}(x,k))\\
&&+\int_0^t\int_{E\setminus (B_{k\!+\!1}\cup C)}\!\!\tilde{Q}((dy,k\!+\!1),du|(x,k),\tilde{f}(x,k))
W_n(y,k+1,t-u)\leq W_{n+1}(x,k,t).
\end{eqnarray*}
Hence, (\ref{eq4.13}) holds. Letting $n\to \infty$ in (\ref{eq4.13}) and noting $\tilde{\sigma}_n\uparrow \infty$ under $\tilde{f}$, yield that $\tilde{G}(x,k,t,\tilde{f})\leq W^*(x,k,t)\ (\tilde{f}\in \tilde{\Pi}_{sd})$.
Taking maximum over $\tilde{f}\in \tilde{\Pi}_{sd}$ yields $W^*(x,k,t)=\tilde{G}^*(x,k,t)$.
\end{proof}
From Theorem~\ref{thm-W}, we can consider to iterate the sequence $\{W_n(x,k,t): n \geq 1\}$ for all $(x,k,t)\in (E \setminus C) \times \mathbb{Z}_+ \times \mathbb{R}_T$, and then obtain the approximation of the maximal reach-avoid probability. To ensure the convergence of the following improved value-type algorithm, we present Proposition \ref{thm-convergence} as below.

\begin{pro}\label{thm-convergence}
{\rm Suppose that (\ref{As-Q}) holds. Let $\{W_n(x,k,t): n \geq 1\}$ be defined by (\ref{W-2}), $\beta:=(1-\epsilon_0^{\tilde{K}})^{\frac{1}{\tilde{K}}}$, where $\tilde{K}$ is given in the proof of Proposition~\ref{th-unique}.
	\begin{description}
\item[(a)] For given $t\in \mathbb{R}_T$ and any sufficiently small $\rho>0$, take $\tilde{l}:=\tilde{K}+\log_{\beta}\rho$. Then,
\begin{eqnarray*}
	0 \leq G^*(x,k,t)-W_{n_{\tilde{l}}}(x,k,t) \leq \rho \ \ \text{for  all }  (x,k) \in (E\! \setminus\! C)\! \times\! \mathbb{Z}_+.
\end{eqnarray*}
	
\item[(b)] For given $t\in \mathbb{R}_T$ any $\epsilon>0$, there exists an integer $n_{\tilde{l}}$ and a policy $\tilde{f}^{*(t)} \in \tilde{\Pi}_{sd}$ such that $W_{n_{\tilde{l}}+1}(x,k,t)=\mathcal{L}^{\tilde{f}^{*(t)}}
    W_{n_{\tilde{l}}}(x,k,t)$ for  all  $(x,k) \in (E \setminus C) \times \mathbb{Z}_+$, and $\tilde{f}^{*(t)}$ is an $\epsilon$-optimal policy for horizon $t$.
\end{description}
}
\end{pro}	
\begin{proof}
By the proof of Theorem~\ref{thm-W}, there exists a policy $\hat{f} \in  \tilde{\Pi}_{sd}$ with $\hat{f}(x,k)=a^{*}_{x,k}$ for all $(x,k) \in (E \setminus C) \times \mathbb{Z}_+$ such that $
		\max\limits_{f \in \tilde{\Pi}_{sd}}\mathcal{L}^{f}W_{n_l}(x,k,t)=\mathcal{L}^{
		\hat{f}}W_{n_l}(x,k,t)$.
	Then,
\begin{eqnarray*}
&&W^*(x,k,t)-W_{n_{_{l+1}}}(x,k,t)
\leq W^*(x,k,t)-W_{n_{l}+1}(x,k,t)\\
&=& \mathcal{L}^{\hat{f}}W^*(x,k,t)-\max\limits_{f \in \tilde{\Pi}_{sd}}\mathcal{L}^{f}W_{n_l}(x,k,t)
= \mathcal{L}^{\hat{f}}[W^*(x,k,t)-W_{n_l}(x,k,t)].
\end{eqnarray*}
	Therefore, by the definition of $F_{\delta}(t)$ given in the proof of Proposition \ref{th-unique} and an induction argument, we have $W^*(x,k,t)-W_{n_l}(x,k,t) \leq F_{\delta}^{*(l)}(t)\ (l \geq 0)$. Hence, noting $F_{\delta}^{*(l)}(t)\leq (1-\epsilon_0^{\tilde{K}})^{\lfloor\frac{l}{\tilde{K}}\rfloor}$ yields that for all $(x,k) \in (E \setminus C) \times \mathbb{Z}_+$, $W^*(x,k,t)-W_{n_l}(x,k,t) \leq \beta^{\lfloor\frac{l}{\tilde{K}}\rfloor\tilde{K}}<\rho\ \ (l\geq \tilde{K}+\log_{\beta}\rho)$. Take $\tilde{l}=\tilde{K}+\log_{\beta}\rho$. Since $G^*(x,k,t)=W^*(x,k,t)$, we obtain that $G^*(x,k,t)-W_n(x,k,t)\leq \rho$ for $n\geq n_{\tilde{l}}$.

 We now consider (b). Take $\rho=\frac{\epsilon}{2}$ and $\tilde{l}:=\tilde{K}+\log_{\beta}\rho$. By the proof of (a) and the proof of Theorem~\ref{thm-W}, there exists $\hat{f} \in \tilde{\Pi}_{sd}$ such that $W_{n_{\tilde{l}}+1}(x,k,t)=\mathcal{L}^{\hat{f}}W_{n_{\tilde{l}}}(x,k,t)$ for all $(x,k) \in (E \setminus C) \times \mathbb{Z}_+$. From Proposition \ref{th-unique} and its proof, we know that
\begin{eqnarray*}
	\tilde{G}(x,k,t,\hat{f})-W_{n_{\tilde{l}}+1}(x,k,t)
=\tilde{\mathcal{L}}^{\hat{f}}[\tilde{G}(x,k,t,\hat{f})
	-W_{n_{\tilde{l}}}(x,k,t)]
<\beta^{\lfloor\frac{l}{\tilde{K}}\rfloor\tilde{K}}.
\end{eqnarray*}
Take $
\tilde{f}^{*(t)}=\hat{f}$. Therefore, $\tilde{G}(x,k,t,\tilde{f}^{*(t)})-W_{n_{\tilde{l}}+1}(x,k,t)
<\beta^{\lfloor\frac{l}{\tilde{K}}\rfloor\tilde{K}}$.
Hence, by (a), we finally get that
$|\tilde{G}^*(x,k,t)-\tilde{G}(x,k,t, \tilde{f}^{*(t)})|<\rho+\beta^
{\lfloor\frac{l}{\tilde{K}}\rfloor\tilde{K}}\leq\epsilon$,
which completes the proof.
\end{proof}

Based on Lemma \ref{lem3.1}, Theorem
 \ref{thm3.1}, Theorem \ref{OE-optimal}, Theorem \ref{thm-W} and Proposition~\ref{thm-convergence}, we obtain an algorithm through an improved value iterative-type to approach to the maximal reach-avoid probability $(G^*(x,T): x \in E\! \setminus\! C)$ and an $\epsilon$-optimal policy $\pi^*$. This algorithm only consider one value at every iteration. Precisely, take $\tilde{l}=\tilde{K}+\log_{\beta}\rho$ and find $n_{\tilde{l}}$ by $\mathcal{L}^{\hat{f}}W_{n_l}(x,k,T)=\max\limits_{a \in A(x,k)}\mathcal{L}^aW_{n_l}(x,k,T)\ (l\geq 1)$. Let
\begin{eqnarray*}
\begin{cases}
	&W_0(x,n_{\tilde{l}},T)=\max\limits_{a^{(0)} \in A(x,n_{\tilde{l}})}\tilde{Q}((C,n_{\tilde{l}}+1),T|(x,n_{\tilde{l}})
,a^{(0)}),\\
	&W_{\tilde{n}}(x,n_{\tilde{l}}-\tilde{n},T)
	=\max\limits_{a^{(\tilde{n})} \in A(x,n_{\tilde{l}}-\tilde{n})}\mathcal{L}^{a^{(\tilde{n})}}
W_{\tilde{n}-1}(x,n_{\tilde{l}}-\tilde{n},T),\quad \tilde{n} \geq 1,
	\end{cases}
\end{eqnarray*}
where $\beta:=(1-\epsilon_0^{\tilde{K}})^{\frac{1}{\tilde{K}}}$ and $\tilde{K}$ is given in the proof of Proposition~\ref{th-unique}. We find that for all $x \in E \setminus C$, when step $\tilde{n}=n_{\tilde{l}}$, we get $W_{n_{\tilde{l}}}(x,0,T)$, which is the approximate value of the maximal reach-avoid probability, i.e.,
\begin{eqnarray}\label{iteration}
	 W_0(x,n_{\tilde{l}},T) \Rightarrow  W_{n_{\tilde{l}}}(x,0,T)\approx\tilde{G}^*(x,0,T)=G^*(x,T).
\end{eqnarray}


\begin{alg}\label{al-finite} \rm{Assume that $t=T$. An improved value iteration algorithm for the $\epsilon$-optimal policy $\pi^*$ and the maximal reach-avoid probability $(G^*(x,T): x \in E\! \setminus\! C)$, is given as below.
	
(1)\ Take $\rho:=\frac{\epsilon}{2}$ and $\tilde{l}=\tilde{K}+\log_{\beta}\rho$. Find $n_{\tilde{l}}$ and $\tilde{f}^*\in \tilde{\Pi}_{sd}$ by $\mathcal{L}^{\tilde{f}^*}W_{n_l}(x,k,T)=\max\limits_{a \in A(x,k)}\mathcal{L}^aW_{n_l}(x,k,T)\ (1\leq l\leq \tilde{l})$. For all $x \in E \setminus C$, let
	\begin{eqnarray*}
		W_0(x,n_{\tilde{l}},T)=\max\limits_{a^{(0)} \in A(x,n_{\tilde{l}})}\tilde{Q}((C,n_{\tilde{l}}+1),T|(x,n_{\tilde{l}})
,a^{(0)})\}.
	\end{eqnarray*}
	
	(2)\ Let $\tilde{n}=1$, and obtain $(W_{\tilde{n}}(x,n_{\tilde{l}},T):\ x \in E \setminus C)$ by
	\begin{eqnarray}\label{W-n}
		W_{\tilde{n}}(x,n_{\tilde{l}}-\tilde{n},T)
		=\max\limits_{a^{(\tilde{n})} \in A(x,n_{\tilde{l}}-\tilde{n})}\mathcal{L}^{a^{(\tilde{n})}}
W_{\tilde{n}-1}(x,n_{\tilde{l}}-\tilde{n},T)
	\end{eqnarray}
	for all $x \in E \setminus C$.

(3)\ If $\tilde{n}=n_{\tilde{l}}$, then stop because $0<\tilde{G}^*(x,0,T)-W_{n_{\tilde{l}}}(x,0,T)<\rho$. Moreover, $(W_{n_{\tilde{l}}}(x,0,T): x\in E \setminus  C)$ is usually regarded as $(\tilde{G}^*(x,0,T): x \in E \setminus C)$, and $\tilde{\pi}^*:=\{\tilde{f}^*, \tilde{f}^*,\cdots\}$ satisfying that for all $x \in E \setminus C$,
\begin{eqnarray}\label{a-op}
	\tilde{f}^*(x,n_{\tilde{l}})=a^*\in \mathop{\arg\max}\limits_{a^{(n_{\tilde{l}})} \in A(x,n_{\tilde{l}})} \mathcal{L}^{a^{(n_{\tilde{l}})}}W_{n_{\tilde{l}}}(x,0,T),
\end{eqnarray}
is an $\epsilon$-optimal policy  of (\ref{newModel}).

(4)\ Set
$\pi^*:=\{\psi^*_n: \ n \geq 0\}$ such that for $n \geq 0$,
\begin{eqnarray}\label{psi*}
\psi^*_n(\cdot|x):=
	\begin{cases}
		\delta_{\tilde{f}^*(x,n)}(\cdot)\ & \ \text{if} \ x \notin B_{n}\\
		g_{n}(\cdot|x)\ & \ \text{if} \ x \in B_{n},
	\end{cases}
\end{eqnarray}
where $\{g_n(\cdot |x): n \geq 0\}$ is a sequence of probability measures on $A(x)$ for all $x\in E$. Hence, the maximal reach-avoid probability is
\begin{eqnarray*}
G^*(x,T):=\tilde{G}^*(x,0,T)\approx W_{n_{\tilde{l}}}(x,0,T)\ \ \text{for \ all \ }x \in E \! \setminus \! (B_0 \cup C),
\end{eqnarray*}
and $\pi^*=\{\psi^*_n: \ n \geq 0\}$ defined by (\ref{psi*}) is an $\epsilon$-optimal policy  of (\ref{Model}).
}
\end{alg}

\section{Plane flight example}\label{sec-5}
In the final section, we give an example to illustrate potential situations in which our model can be applied, and the following plane flight example is already analyzed in \cite{IAC-05}, which computed the maximal reachable set.

\begin{exm}\label{exm5.1}
\rm{ Continue with Example~\ref{exm2.1}. Below we give three different situations of obstacle sets:
	\begin{eqnarray}\label{B1}
		B^1_k:=\{0\},\ k \geq 0;
	\end{eqnarray}
	\begin{eqnarray}\label{B2}
		B^2_k:=\{1\},\  k \geq 0;
	\end{eqnarray}
	\begin{eqnarray}\label{B3}
		B^4_k:=\begin{cases}
			\{0\},\ \ k\ \text{is\ odd},\\
			\{1\},\ \ k\ \text{is\ even}.
		\end{cases}
	\end{eqnarray}
The corresponding transition kernel is defined as below: for all $i \in E\! \setminus\! C$,
\begin{eqnarray*}
\begin{cases}
Q(j,t|i,\alpha):=\begin{cases}
\frac{t}{\mu(i,\alpha)}p(j|i,\alpha),\ & \ 0 \leq t \leq \mu(i,\alpha),\\
p(j|i,\alpha),\ & \ t>\mu(j,\alpha);
\end{cases}\\
Q(j,t|i,\beta):=\begin{cases}
\frac{t}{\mu(i,\beta)}p(j|i,\beta),\ & \ 0 \leq t \leq \mu(i,\beta),\\
p(j|i,\beta),\ & \ t>\mu(i,\beta);
\end{cases}\\
Q(j,t| i, \gamma):=(1-e^{-\mu(i,\gamma)t})p(j|i,\gamma),
\end{cases}
\end{eqnarray*}
where $p(j|i,a)$ for all $a \in A(i)$ is given by Table 1.
Therefore, under the above transition kernel, our purpose is
computing the maximal reach-avoid probability of the vehicle to target $C$ within finite time $T$, i.e., $G^*(i,T):=\sup_{\psi \in \Pi_{rm}}G(i,T,\psi)$, and
finding the optimal policy $\psi^* \in \Pi_{rm}$ such that $G(i,T,\psi^*)=G^*(i,T)$ for all $i \in E \setminus (B^s_0 \cup C)\ (s=1,2,3)$. Therefore, under the above transition kernel, our purpose is
computing the maximal reach-avoid probability of the vehicle to target $C$ within finite time $T$, i.e., $G^*(i,T):=\sup_{\psi \in \Pi_{rm}}G(i,T,\psi)$, and
finding the optimal policy $\psi^* \in \Pi_{rm}$ such that $G(i,T,\psi^*)=G^*(i,T)$ for all $i \in E \setminus (B^s_0 \cup C)\ (s=1,2,3)$.

From the description above, we obtain the process with semi-Markov kernel given above. By Theorem \ref{thm3.1}, it is natural to consider the equivalent model (\ref{newModel}), where in two situations (\ref{B1})-(\ref{B3}), the new state space is $S:=E\times \mathbb{Z}_+$, the new obstacle sets are  $\tilde{B}^s:=\cup_{k=0}^{\infty}B^s_k \times \{k\}$ with $s=1,2,3$, respectively, the new target set is $\tilde{C}:=C \times \mathbb{Z}_+$, the new action space is composed of $A(i,k):=\{\alpha, \beta, \gamma\}$ for all $(i,k) \in S \setminus (\tilde{B}^s\cup \tilde{C})\ (s=1,2,3)$, and $A(i,k):=\{\Delta^*\}$ for all $(i,k) \in \tilde{B}^s\ (s=1,2,3)$, where there is no transition from state $(i,k) \in \tilde{B}^s\ (s=1,2,3)$ under action $\Delta^*$, and the new transition kernel is given as below: for all $(i,k) \in S\setminus (\tilde{B}^s\cup \tilde{C})\ (s=1,2,3)$,
\begin{eqnarray*}
\begin{cases}
\tilde{Q}((j,k+1),t|(i,k),\alpha):=\begin{cases}
\frac{t}{\mu(i,\alpha)}p(j|i,\alpha),\ & \ 0 \leq t \leq \mu(i,\alpha),\\
p(j|i,\alpha),\ & \ t>\mu(i,\alpha);
\end{cases}\\
\tilde{Q}((j,k+1),t|(i,k),\beta):=\begin{cases}
\frac{t}{\mu(i,\beta)}p(j|i,\beta),\ & \ 0 \leq t \leq \mu(i,\beta),\\
p(j|i,\beta),\ & \ t>\mu(i,\beta);
\end{cases}\\
				\tilde{Q}((j,k+1),t| (i,k), \gamma):=(1-e^{-\mu(i,\gamma)t})p(j|i,\gamma).
\end{cases}
\end{eqnarray*}
Then, from Theorem \ref{thm3.1}, we only need to calculate $\tilde{G}^*(x,0,T):=\sup_{\tilde{\psi} \in \tilde{\Pi}_s}\tilde{G}(x,0,T,\tilde{\psi})$ and find the equivalent optimal policy $\tilde{\psi}^* \in \tilde{\Pi}_s$. To take numerical calculation for this example, we assume that the states are simplified as $0,1,2,3,4$, which denote five different longitudinal axis positions of the vehicle. Moreover, we assume that $T=18$, $\tilde{B}^1:=\{(0,k):k\geq 0\}$, $\tilde{B}^2:= \{(1,k):k\geq 0\}$, $\tilde{B}^3:=\{(0,k):k\  \text{is odd}\} \cup \{(1,k):k  \ \text{is  even}\}$ and $\tilde{C}:=\{4\} \times \mathbb{Z}_+$. The data of the model is given by Table 1.

\begin{table}[!htbp]
\centering
\caption{The data of the model}
\begin{tabular}{|c|c|c|c|c|c|c|c|}
\hline
state $i$&action $a$&$\mu(i,a)$&$p(0|i,a)$&$p(1|i,a)$&$p(2|i,a)$&$p(3|i,a)$&$
				p(4|i,a)$\\
\hline
\multirow{3}*{0}&$\alpha$&20&0&0.2&0.3&0.2&0.3\\
\cline{2-8}
\multirow{3}*{}&$\beta$&19&0&0.3&0.1&0.2&0.4\\
\cline{2-8}
\multirow{3}*{}&$\gamma$&21&0&0.3&0.2&0.2&0.3\\
\hline
\multirow{3}*{1}&$\alpha$&20&0.2&0&0.3&0.1&0.3\\
\cline{2-8}
\multirow{3}*{}&$\beta$&19&0.2&0&0.3&0.1&0.4\\
\cline{2-8}
\multirow{3}*{}&$\gamma$&21&0.3&0&0.3&0.1&0.3\\
\hline
\multirow{3}*{2}&$\alpha$&22&0.05&0.4&0&0.25&0.3\\
\cline{2-8}
\multirow{3}*{}&$\beta$&20&0.05&0.3&0&0.3&0.35\\
\cline{2-8}
\multirow{3}*{}&$\gamma$&19&0.1&0.2&0&0.4&0.3\\
\hline
\multirow{3}*{3}&$\alpha$&19&0.05&0.35&0.2&0&0.4\\
\cline{2-8}
\multirow{3}*{}&$\beta$&18&0.05&0.35&0.3&0&0.3\\
\cline{2-8}
\multirow{3}*{}&$\gamma$&22&0.05&0.3&0.3&0&0.35\\
\hline
\multirow{3}*{4}&$\alpha$&22&0.3&0.2&0.2&0.3&0\\
\cline{2-8}
\multirow{3}*{}&$\beta$&20&0.2&0.3&0.3&0.2&0\\
\cline{2-8}
\multirow{3}*{}&$\gamma$&19&0.4&0.1&0.1&0.4&0\\
\hline
\end{tabular}
\end{table}

\begin{pro}
\rm{Under the above assumption, the explicit maximal reach-avoid probability of original model (\ref{Model}) and the specific $\epsilon$-optimal policy are obtained, where the $\epsilon$-optimal policy is indeed affected by horizons.
}
\end{pro}
\begin{proof}
Indeed, Assumption \ref{As-2} holds with $\delta=1$ and $\epsilon_0=\frac{17}{18}$ by verifying Proposition \ref{As-2}. Choose $\epsilon=1.02\times 10^{-5}$ and we get $\tilde{K}=6$ and $\rho=5.1 \times 10^{-6}$. Thus, $n_{\tilde{l}}=8$. By Lemma \ref{lem3.1} and Theorem \ref{OE-optimal}, the existence of the $\epsilon$-optimal policy is ensured.
\end{proof}
Now we calculate the approximate value of $\tilde{G}^*(i,0,18)$ for $i=0,1,2,3,4$ by MATLAB software, that is, $W^1_8(i,0,18)$, $W^2_8(i,0,18)$ and $W^3_8(i,0,18)$ for $i=0,1,2,3,4$ in situations (\ref{B1})-(\ref{B3}), where the approximation calculation in step 2 of the integrals is from the numerical integration method.
Hence, by step 4, we obtain that the maximal reach-avoid probability $(G^*(i,18):\ i \in \{0,1,2,3\})$ in situations (\ref{B1})-(\ref{B3}) are approximately given as below, respectively:
\begin{eqnarray*}
\begin{cases}
	W^1_8(0,0,18)=0\\
	W^1_8(1,0,18)\approx0.648729\\
	W^1_8(2,0,18)\approx0.734136\\
	W^1_8(3,0,18)\approx0.752358.
\end{cases}
\quad\quad
	\begin{cases}
	W^2_8(0,0,18)\approx0.523974\\
	W^2_8(1,0,18)=0\\
	W^2_8(2,0,18)\approx0.556912\\
	W^2_8(3,0,18)\approx0.506937,
\end{cases}
\quad\quad
	\begin{cases}
	W^3_8(0,0,18)=0.638189\\
	W^3_8(1,0,18)=0,\\
	W^3_8(2,0,18)\approx0.66149\\
	W^3_8(3,0,18)\approx0.661607.
\end{cases}
\end{eqnarray*}		
and the $\epsilon$-optimal policy in all three situations is $\pi^*:=\{\psi^*_n: \ n \geq 0\}$ satisfying that for $n \geq 0$,
\begin{eqnarray*}
	\psi^*_n(\cdot|x):=
	\begin{cases}
		\delta_{\beta}(\cdot)\ & \ \text{if} \ x \notin B_{n}\\
		g_{n}(\cdot|x)\ & \ \text{if} \ x \in B_{n}.
	\end{cases}
\end{eqnarray*}
We give the situation of $W^1_8(i,0,t)$, $W^2_8(i,0,t)$ and $W^3_8(i,0,t)$ for all $i \in \{0,1,2,3\}$  with respect to $t \in [0,18]$ in Figure 1, Figure 2 and Figure 3, respectively.

\begin{center}
\includegraphics[height=0.36\textwidth]{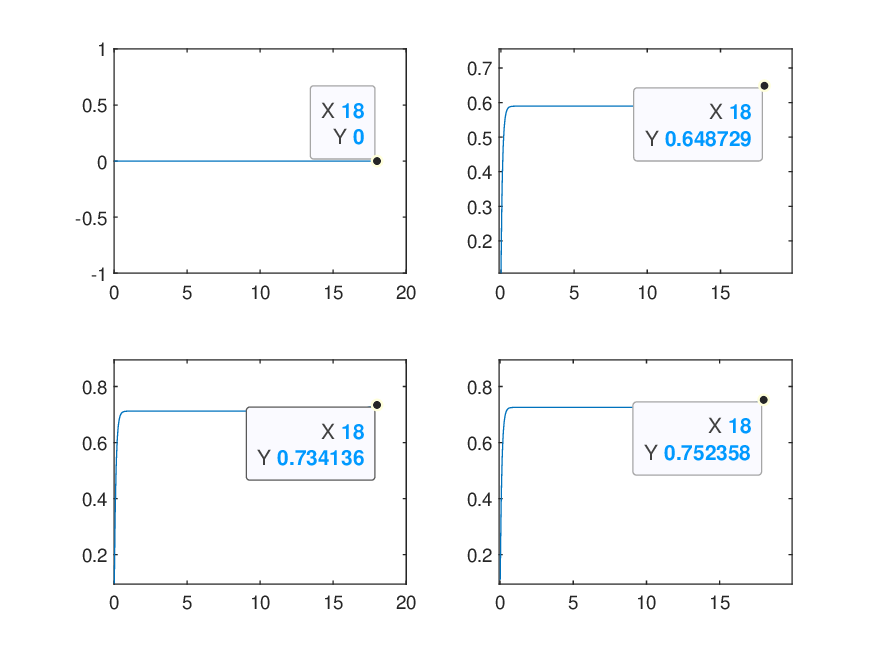}\quad \includegraphics[height=0.36\textwidth]{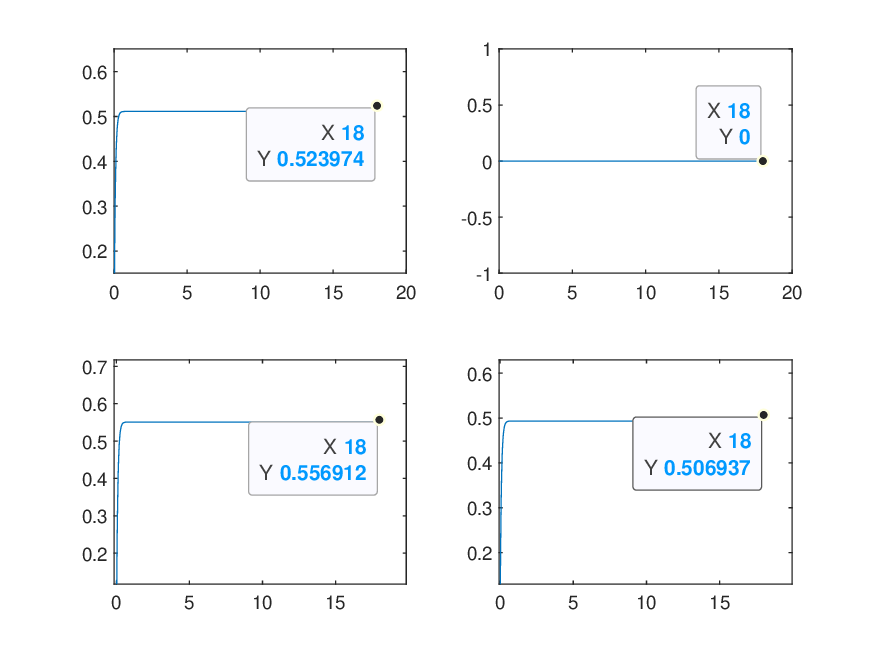}
{\tiny\text{Fig 1:\ The values of $W^1_8(i,0,t)$ with respect to $t \in [0,18]$.\quad \quad \quad \quad \quad \quad Fig 2:\ The values of $W^2_8(i,0,t)$ with respect to $t \in [0,18]$.}}
\end{center}
\begin{center}
	 \includegraphics[height=0.36\textwidth]{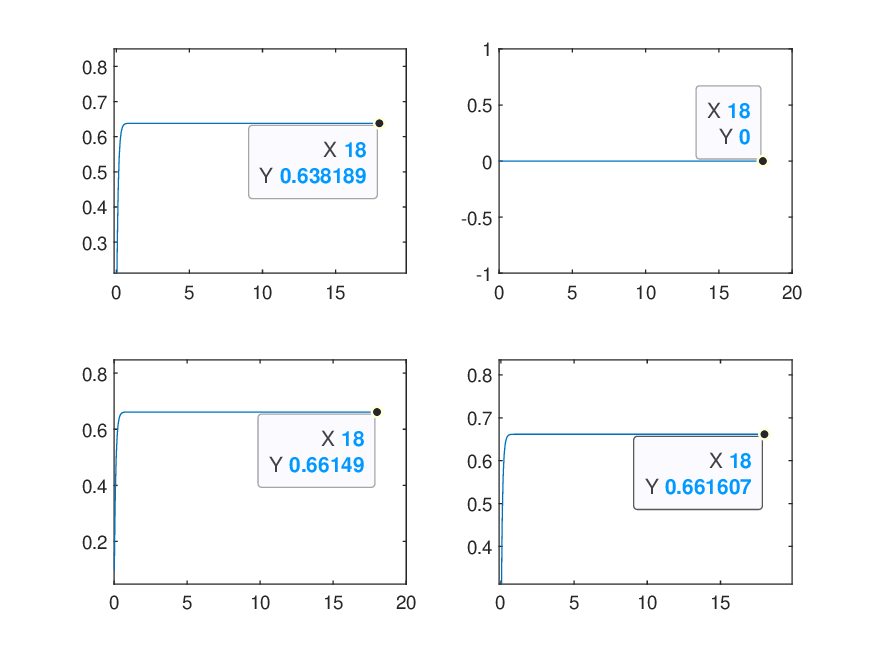}
	{\text{ Fig 3:\ The values of $W^3_8(i,0,t)$ with respect to $t \in [0,18]$.}}
\end{center}

\begin{rem}\label{rem-3}
	{\rm	By Figures 1-2, we see that in the fixed obstacle set case, when the transition probability from regular states (that is, states in $E \setminus (B_0 \cup C)$) to the obstacle set is smaller, the maximal reach-avoid probability bigger. However, based on situation (\ref{B2}), we change the obstacle state $1$ to $0$ at decision epochs $3,4,5,6$ and obtain situation (\ref{B3}) (i.e., varying obstacle set case), it can be seen that $W^3_8(i,0,18)>W^2_8(i,0,18)\ (i\neq 1)$, see Figure 3. Therefore, based on the second situation, in order to enlarge the maximal reaching probability, we only need to suitably change the obstacle set at finite decision epochs (since $n_{\tilde{l}}=8$).}
\end{rem}

}

\end{exm}

\end{document}